\documentclass[reqno]{amsart}
\usepackage{graphics}
\usepackage{latexsym}
\usepackage{amsfonts}
\usepackage{amsmath}
\usepackage{amssymb}
\usepackage{amsthm}
\usepackage{multicol}
\usepackage{dsfont}
\usepackage{enumerate}
\numberwithin{equation}{section}
\newtheorem{theorem}{Theorem}[section] 
\newtheorem{proposition}{Proposition}[section] 
\newtheorem{corollary}{Corollary}[section] 
\newtheorem{lemma}{Lemma}[section] 
\newtheorem{remark}{Remark}[section] 

\newcommand{\R}{\mathbb R} 
\newcommand{\C}{\mathbb C} 

\newcommand{\E}{\mathbb{E}}

\renewcommand{\P}{\mathbb{P}}

\newcommand{\var}{\mathbb{V}}
\newcommand{\V}{\mathbb{V}}

\newcommand{\Tr}{\mathrm{Tr}}
\newcommand{\tr}{\mathrm{tr}}
\newcommand{\diag}{\text{diag}}
\newcommand{\mb}{\mathbf}

\renewcommand\Re{\operatorname{Re}}
\renewcommand\Im{\operatorname{Im}}
\newcommand{\indicator}[1]{\ensuremath{\mathbf{1}_{\{#1\}}}}

\renewcommand{\P}{\mathbb P}

\begin{document}
\title[Finite Rank Deformations]{On Finite Rank Deformations of Wigner Matrices II: Delocalized Perturbations}
\author[D. Renfrew]{David Renfrew} \thanks{D.R. has been supported in part by the NSF grants VIGRE DMS-0636297, DMS-1007558, and 
DMS-0905988 }
\address{Department of Mathematics, University of California, Davis, One Shields Avenue, Davis, CA 95616-8633  }
\email{drenfrew@math.ucdavis.edu}
\author[A. Soshnikov]{Alexander Soshnikov}
\address{Department of Mathematics, University of California, Davis, One Shields Avenue, Davis, CA 95616-8633  }
\thanks{A.S. has been supported in part by the NSF grant DMS-1007558}
\email{soshniko@math.ucdavis.edu}
\begin{abstract}
We study the distribution of the outliers in the spectrum of finite rank deformations of Wigner random matrices.
We assume that the matrix entries have finite fourth moment and extend the results by Capitaine, Donati-Martin, and F\'eral for perturbations whose eigenvectors are delocalized.
\end{abstract}

\maketitle

\section{Introduction}
In this paper we continue the study of the eigenvalues of finite rank deformations to Wigner random matrices, extending the results of \cite{PRS}
to a larger class of perturbations. 

Let $\mb X_N:= \frac{1}{\sqrt{N}} \mb W_N $ be a random Wigner real symmetric (Hermitian) matrix.
In the real symmetric case, we assume that the entries 
$$(\mb W_N)_{jk},\ 1\leq j\leq k \leq N,$$ 
are independent random variables
such that the off-diagonal entries satisfy
\begin{equation}
\label{offdiagreal}
\E [(\mb W_N)_{jk}]=0, \ \V[(\mb W_N)_{jk}]=\sigma^2, \ 1\leq j<k \leq N, \ m_4:=\sup_{j\not=k,N} \E[ (\mb W_N)_{jk}^4]<\infty,
\end{equation}
and the Lindeberg type condition for the fourth moments takes place,
\begin{equation}
\label{lind1}
L_N(\epsilon) \to 0, \ \text{as} \ N\to \infty, \ \forall \epsilon>0, 
\end{equation}
where 
\begin{equation}
\label{lind2}
L_N(\epsilon)= \frac{1}{N^2}\* \sum_{1\leq i<j\leq N} \E\left( |(\mb W_N)_{ij}|^4 \*\indicator{|(\mb W_N)_{ij}|\geq \epsilon\*N^{1/4}}\right ).
\end{equation}
Here and throughout the paper, $\E \xi $ denotes the mathematical expectation and $\V \xi $ the variance of a random variable $\xi.$ 
In addition, we assume that the diagonal entries  
satisfy
\begin{align}
\label{diagreal}
& \E [(\mb W_N)_{ii}]=0, \ \ 1\leq i \leq N, \ \sigma^2_1:=\sup_{i,N} \E [(\mb W_N)_{ii}^2] < \infty,\\
\label{diagreal1}
& l_N(\epsilon)\to 0, \ \text{as} \ N\to \infty, \ \forall \epsilon>0, \ \text{where} \\
\label{diagreal2}
& l_N(\epsilon)= \frac{1}{N}\* \sum_{1\leq i \leq N} \E\left( |(\mb W_N)_{ii}|^2 \*\indicator{|(\mb W_N)_{ii}|\geq \epsilon\*\sqrt{N}}\right ).
\end{align}

We note that (\ref{lind1}) and (\ref{diagreal1})  are satisfied if there exist an $\epsilon > 0$ such that
\begin{equation}
\label{foureps}
\sup_{i\not=j,N} \E [(\mb W_N)_{ij}^{4+\epsilon}]<\infty, \ \sup_{i,N} \E [(\mb W_N)_{ii}^{2+\epsilon}]<\infty.
\end{equation}
If the off-diagonal elements $(\mb W_N)_{jk}$ are identically distributed Gaussian and the diagonal elements $(\mb W_N)_{ii}$ are also identically 
distributed Gaussian with twice the variance of the off-diagonal elements then $\mb W_N$ is said to belong to the Gaussian Orthogonal Ensemble (GOE).

In the Hermitian case, we assume that the entries  
$$\Re (\mb W_N)_{jk}, \ \Im (\mb W_N)_{jk}, \ 1\leq j<k \leq N, \ (\mb W_N)_{ii}, \ 1\leq i \leq N,$$
are independent random variables such that
the off-diagonal entries satisfy
\begin{gather}
\label{offdiagherm1}
 \E \Re (\mb W_N)_{jk}= \E \Im (\mb W_N)_{jk} = 0, \ \ 1\leq j<k \leq N, \\
\label{offdiagherm2}
 \V  \Re (\mb W_N)_{jk} = \V \Im (\mb W_N)_{jk}= \frac{\sigma^2}{2}, \ 1\leq j<k \leq N, \ m_4:=\sup_{j\not=k, N} \E |(\mb W_N)_{jk}|^4<\infty,
\end{gather}
and the Lindeberg type condition (\ref{lind1}) for the fourth moments of the off-diagonal entries takes place.
In addition, we assume that the diagonal entries  satisfy
\begin{equation}
\label{diagherm}
\E (\mb W_N)_{ii}=0, \ \ 1\leq i \leq N, \ \sigma^2_1:=\sup_{i,N} \E |(\mb W_N)_{ii}|^2<\infty,
\end{equation}
and the Lindeberg type condition (\ref{diagreal1}) for the second moments of the diagonal entries takes place.

If the real and imaginary parts of the off-diagonal elements $(\mb W_N)_{jk}$ are independent 
identically distributed Gaussian random variables and the diagonal elements $(\mb W_N)_{ii}$ 
are also identically distributed Gaussian random variables
with twice the variance of the real part of the off-diagonal entries then $\mb W_N$ is said to belong to the Gaussian 
Unitary Ensemble (GUE).

We refer the reader to \cite{AGZ}, \cite{B}, \cite{BG}, and \cite{M} for basic results about standard real symmetric and Hermitian Wigner matrices.
In particular, the Wigner semicircle law states that the empirical distribution of the eigenvalues of $\mb X_N= \frac{1}{\sqrt{N}} \mb W_N $ converges
as $N \to \infty$ to the nonrandom limiting probability distribution $\mu_{sc},$ known as the semicircle distribution, 
whose density with respect to the Lebesgue measure is given by
\begin{equation}
\label{polukrug}
\frac{d \mu_{sc}}{dx}(x) := \frac{1}{2 \pi \sigma^2} \sqrt{ 4 \sigma^2 - x^2} \mathbf{1}_{[-2 \sigma , 2 \sigma]}(x).
\end{equation}
The Stieltjes transform of the semicircle distribution 
\begin{equation}
\label{steltsem} 
g_\sigma(z) := \int \frac{d \mu_{sc}(x)}{z-x}= \frac{z-\sqrt{z^2-4\*\sigma^2}}{2\*\sigma^2}, \ z \in \C \backslash [-2\*\sigma, 2\*\sigma].
\end{equation}
is the solution to
\begin{equation}
\sigma^2 g_\sigma^2(z) - z g_\sigma(z) + 1 =0 
\end{equation}
that decays to $0$ as $z \to \infty$.

We consider the spectrum of $\frac{1}{\sqrt{N}}\* \mb W_N + \mb A_N= \mb X_N + \mb A_N$ where 
$\mb A_N$ is a deterministic real symmetric (Hermitian) matrix of fixed finite rank 
$r$. 
Spectral properties of finite rank perturbations of Wigner matrices have been studied extensively since the pioneering paper \cite{FK} by F\"uredi and  
Koml\'os who considered $(\mb A_N)_{ij}=\frac{c}{\sqrt{N}}, \ 1\leq i, j \leq N,$ which corresponds to the case of a Wigner matrix with non-centered entries having mathematical expectation $c,$  where $c$ is a fixed non-zero real number. It was shown that the largest eigenvalue of $\mb X_N + \mb A_N$
is asymptotically normal with mathematical expectation $c\* \sqrt{N} + \frac{\sigma^2}{c\*\sqrt{N}}$ and variance $\frac{2\*\sigma^2}{N}$ 
in the real symmetric case and $\frac{\sigma^2}{N}$ in the Hermitian case.

The more mathematically challenging case when $c=\frac{\theta}{\sqrt{N}}$ was studied in \cite{P} (in the GUE case) and \cite{FP} 
(for arbitrary Hermitian Wigner matrices with symmetrical sub-Gaussian marginal distribution). 
In particular, it was shown that there is a phase transition at 
$\theta=\sigma.$ For $\theta>\sigma$ the spectrum of $\mb X_N + \mb A_N$ has one outlier which is asymptotically normal with mathematical expectation
$\rho=\theta + \frac{\sigma^2}{\theta}$ and variance $\frac{\sigma^2(\theta^2-\sigma^2)}{\theta^2\*N}.$ 
For $0<\theta<\sigma,$ the largest eigenvalue of $\mb X_N + \mb A_N$ fluctuates on the scale $N^{-2/3}$ around $2\*\sigma,$ and has Tracy-Widom 
distribution in the limit.  Large deviations for the outlier
in the Gaussian case were studied by Maida in \cite{Mai}.

The case of arbitrary finite rank perturbations has been considered by P\'ech\'e (\cite{P}) for a GUE matrix and by several other authors
for real symmetric and Hermitian Wigner random matrices
(see e.g. \cite{CDF1}, \cite{CDF},   \cite{CDFF}, \cite{BR}, \cite{BGM}, \cite{BGM1}, \cite{PRS}, \cite{KYlsc} and references therein).
We also refer the reader to \cite{BW1}, \cite{BW2}, \cite{BV1}, and \cite{BV2} for related results for unitary and orthogonal ensembles of 
random matrices. Finally, we note several results about the outliers in the spectrum of spiked sample covariance random matrices 
(\cite{Jo}, \cite{BBP}, \cite{Paul}, \cite{BS}) and non-Hermitian random matrices (\cite{T}).

Let us denote the ordered eigenvalues of $\mb A_N$ by $\theta_1 > \ldots > \theta_J.$  The multiplicity of $\theta_j$ is fixed and denoted by
$k_j, \ 1\leq j\leq J$.
We assume that both the eigenvalues of $\mb A_N$ and their multiplicities are independent of $N.$ 
Let $j_0$ be such that $\theta_{j_0} = 0$. Thus, $\mb A_N$ has $j_0-1$ distinct positive eigenvalues (not counting multiplicities).
The ordered eigenvalues of $\mb X_N + \mb A_N$ are denoted 
$\lambda_1 \geq \ldots \geq \lambda_N $. 
Our first theorem is:

\begin{theorem}[\cite{CDF1}, \cite{PRS}]
\label{firsttheorem}
Let $\mb X_N=\frac{1}{\sqrt{N}} \mb W_N$ be a random real symmetric (Hermitian) Wigner matrix
defined in (\ref{offdiagreal}-\ref{diagreal2}) (respectively (\ref{offdiagherm1}-\ref{diagherm})). Let $\mb A_N$ be a deterministic 
real symmetric (Hermitian) matrix of fixed finite rank $r$ as above.
Let $J_{+\sigma}$ (resp. $J_{-\sigma}$) be the number of $j$'s such that $\theta_j > \sigma$ (resp. $\theta_j < -\sigma$) and let 
\begin{equation}
\label{minsk}
\rho_{\theta_j} = \rho_j := \theta_j + \frac{\sigma^2}{\theta_j}.
\end{equation}
Then the following holds:
\begin{enumerate}[(a)]
\item For $1 \leq j \leq J_{+\sigma}, 1 \leq i \leq k_j, \lambda_{k_1 + \ldots + k_{j-1}+i} \to \rho_{j},$ 
\item$ \lambda_{k_1 + \ldots + k_{J_{+\sigma}}+1} \to 2 \sigma,$ 
\item $\lambda_{k_1 + \ldots + k_{J-J_{-\sigma}}} \to - 2 \sigma,$
\item For $ j \geq J- J_{-\sigma}+1, 1 \leq i \leq k_j, \lambda_{k_1 + \ldots + k_{j-1}+i} \to \rho_{j}.$
\end{enumerate}
The convergence in (a)-(d) is in probability.
\end{theorem}

In \cite{CDF1}, Capitaine, Donati-Martin, and F\'eral consider Wigner matrices with i.i.d. entries 
whose marginal distribution of the matrix entries of 
$\mb W_N$ is symmetric and satisfies a 
Poincar\'e inequality \eqref{poin}. Under these conditions they show that the convergence in Theorem \ref{firsttheorem} takes place almost surely.
We recall that a probability measure  
$\P$ on $\R^M$ satisfies a Poincar\'e inequality with constant $\upsilon>0$ if, for all continuously differentiable 
functions 
$f: \R^M \to \C,$
\begin{equation}
\label{poin}
\V_{\P}(f)= \E_{\P} \left(|f(x)-\E_{\P}(f(x))|^2\right) \leq \frac{1}{\upsilon}\* \E_{\P}[ |\nabla f(x)|^2 ].
\end{equation} 
It is known that a probability distribution satisfying the Poincar\'e inequality (\ref{poin}) has a subexponential tail (see e.g. \cite{AGZ}).  

In \cite{PRS},  we consider Wigner matrices with 
five finite moments for the off diagonal entries and three finite moments for the diagonal entries. Using the standard truncation argument 
(see the appendix of this paper) these 
conditions can be weakened to those in (\ref{offdiagreal}-\ref{diagreal2}) (respectively (\ref{offdiagherm1}-\ref{diagherm})). 
In this paper we will cite theorems from \cite{PRS} assuming the conditions (\ref{offdiagreal}-\ref{diagreal2}) 
(respectively (\ref{offdiagherm1}-\ref{diagherm})) without further comment. 

The fluctuations of the outlying eigenvalues around $\rho_{\theta_j}$ have been studied in \cite{CDF}, \cite{PRS}, and \cite{KYlsc}.
These fluctuations are dependent on the form of the perturbation.  In particular,
additional assumptions on the eigenvectors of $\mb A_N$ are required for a distributional limit to exist and the limiting distribution of the (properly rescaled) outliers depends on the localized/delocalized nature of the eigenvectors of
$\mb A_N$ corresponding to $\theta_j.$

In \cite{CDF}, the authors consider two regimes:

{\bf Case A} (``The eigenvectors don't spread out'')

The orthonormal eigenvectors of $\mb A_N$ corresponding to $\theta_j$ are spanned by a finite number 
$K_j$ of canonical basis vectors of $\C^N$ (without loss 
of generality we can assume those canonical vectors to be $e_1, \ldots, e_{K_j}$), and the 
(non-zero) coordinates of these eigenvectors are independent of $N$ for all sufficiently large $N.$

{\bf Case B} (``The eigenvectors are delocalized'')

The $l^{\infty}$ norm of every orthonormal eigenvector of $\mb A_N$ corresponding to $\theta_j$ goes to zero as $N \to \infty. $

We denote by $k_{+\sigma} := k_1 +\ldots +k_{J_{+\sigma}}$ the number of positive eigenvalues of $\mb A_N$ bigger than $\sigma$ (counting with 
multiplicities) and
by $k\geq k_{+ \sigma} $ the minimal number of canonical basis vectors $e_1, \ldots, e_N$ of $\C^N$ required to span all the eigenvectors 
corresponding to the eigenvalues $\theta_1, \ldots, \theta_{J_{+ \sigma}}. $

Let us denote
\begin{equation}
\label{ctheta}
c_{\theta_j}:= \frac{\theta_j^2}{\theta_j^2-\sigma^2}.
\end{equation}

The following theorem concerning fluctuations in case A was proved for symmetric marginal distribution satisfying the 
Poincar\'e inequality in \cite{CDF}.   It was extended to the assumptions (\ref{offdiagreal}-\ref{diagreal2}) 
(respectively (\ref{offdiagherm1}-\ref{diagherm})) in \cite{PRS}. 
\begin{theorem} [Theorem 1.3 in \cite{PRS}]
\label{thm:caseA}
Let $\mb X_N=\frac{1}{\sqrt{N}} \mb W_N$ be a random real symmetric (Hermitian) Wigner matrix
defined in (\ref{offdiagreal}-\ref{diagreal2}) (respectively (\ref{offdiagherm1}-\ref{diagherm})).
In Case A, the $k_j$-dimensional vector 
\[ \Big( c_{\theta_j}\* \sqrt{N}\*(\lambda_{k_1+\ldots +k_{j-1}+i} -\rho_j), \ i=1, \ldots, k_j \Big) \]
converges in distribution to the distribution of the ordered eigenvalues of the $k_j \times k_j$ random matrix $V_j$ defined as
\begin{equation}
\label{vj}
\mb V_j:=\mb U^*_j \* (\mb W_j + \mb H_j) \* \mb U_j,
\end{equation}
where $\mb W_j$ is a Wigner random matrix of size $K_j$ with the same marginal distribution of the matrix entries as $\mb W_N,$ 
$\mb H_j$ is a centered Hermitian Gaussian matrix of size $K_j,$ independent of $\mb W_j,$ with independent entries 
$H_{st}, \ 1 \leq s \leq t \leq K_j,$  with
the variance of the entries given by
\begin{align}
& \E(H^2_{ss})=\left(\frac{m_4-(4-\beta)\*\sigma^2}{\theta_j^2}\right) + 
\frac{2}{\beta}\* \frac{\sigma^4}{\theta_j^2-\sigma^2}, \ s=1, \ldots, K_j, \\
& \E(|H_{st}|^2)= \frac{\sigma^4}{\theta_j^2-\sigma^2}, \ 1\leq s<t\leq K_j,
\end{align}
and $\mb U_j$ is a $K_j\times k_j$ such that the ($K_j$-dimensional) columns of $\mb U_j$ are written from the first $K_j$ coordinates of the orthonormal 
eigenvectors
corresponding to $\theta_j.$
\end{theorem}

The next theorem deals with the Case B and is this paper's main result. Before 
stating the theorem we define the matrix of third moments on the off-diagonal by
\begin{equation}
\label{M3}
(\mb M_3)_{ij} := \mu_{3,ij}(1-\delta_{ij}),
\end{equation}
where $\mu_{3,ij} := \E[|W_{ij}|^2W_{ij}]$.
\begin{theorem}
\label{thm:caseB}
Let $\mb X_N=\frac{1}{\sqrt{N}} \mb W_N$ be a random real symmetric (Hermitian) Wigner matrix
defined in (\ref{offdiagreal}-\ref{diagreal2}) (respectively (\ref{offdiagherm1}-\ref{diagherm})). Let $\mb u_N^1, \ldots , \mb u_N^{k_j}$ be a set of orthogonal eigenvectors of $\mb A_N$ with eigenvalue $\theta_j$.
In Case B, the difference between the $k_j$-dimensional vector 
\[ \Big( c_{\theta_j}\* \sqrt{N}\*(\lambda_{k_1+\ldots +k_{j-1}+i} -\rho_j), \ i=1, \ldots, k_j \Big) \]
and the vector formed by the (ordered) eigenvalues of a $k_j\times k_j$ GOE (GUE) matrix with the variance of the matrix 
entries given by $\frac{\theta_j^2\*\sigma^2}{\theta_j^2-\sigma^2}$ plus a deterministic matrix with $lp^{th}$ entry $(1\leq l,p\leq k_j)$ given by 
$\frac{1}{\theta^2\*N} (\mb u_N^l)^* \mb M_3 \mb u_N^p$ converges to zero in probability. 

\end{theorem}
The proof of Theorem \ref{thm:caseB} is in Section \ref{mainthm}.

Theorem \ref{thm:caseB} was originally proved in \cite{CDF} under additional technical assumptions that the entries are i.i.d., their marginal 
distribution is symmetric and satisfies the Poincar\'e inequality, and
$k=o(\sqrt{N}),$  where we recall that $k$ is
the minimal number of canonical basis vectors $e_1, \ldots, e_N$ of $\C^N$ required to span all the eigenvectors 
corresponding to the eigenvalues bigger than $\sigma.$  In \cite{KYlsc}, Knowles and Yin prove Theorem \ref{thm:caseB}
provided $\theta_j$ is a simple eigenvalue.  It should be noted that Knowles and Yin allow the non-zero eigenvalues of $\mb A_N$ to depend on $N$ as 
long 
$$ ||\theta_j|-2\sigma| \geq \phi^{C_1} \*N^{-1/3}, \ \ \min_{j\not=i} |\theta_j-\theta_i|\geq \phi^{C_1} \*N^{-1/2}\* (|\theta_i|-2\*\sigma|)^{-1/2},$$
where $\phi:=(\log N)^{\log \log N},$ and $C_1$ is a positive constant.
In their approach, Knowles and Yin need an additional technical assumption, namely they require 
that the marginal distributions of the entries of $\mb W_N$ are uniformly sub-Gaussian, in a sense that
$$\P(|(\mb W_N)_{ij}|\geq x) \leq d^{-1}\*\exp(-x^d)$$
for some $d>0.$ After this paper was posted online, Knowles and Yin extended their results to arbitrary deterministic matrices with bounded norm and fixed rank in \cite{KYout}.

In addition to the results mentioned above, Knowles and Yin prove the universality of the limit distribution of the largest ``sticking'' 
eigenvalues of
$\mb X_N +\mb A_N$, i.e. the eigenvalues that correspond to $|\theta_j|<\sigma$ (see Theorem 2.7 in \cite{KYlsc}). In other words, they prove that 
the limit distribution of the 1st, 2nd, 3rd, etc largest ``sticking'' eigenvalues of $\mb X_N +\mb A_N$ (i.e. $\lambda_{k_{+\sigma}1},
\lambda_{k_{+\sigma}+2}, \lambda_{k_{+\sigma}+3}$, etc)
coincides with the limit distribution of the 
1st, 2nd, 3rd, etc largest eigenvalues of $\mb X_N.$  We recall that $k_{+\sigma}$ denotes the number of the eigenvalues of $\mb A_N$ greater than 
$\sigma$ (counting with multiplicities).  In the Gaussian case,
the limiting distribution of the largest eigenvalues of a GUE (GOE) random matrix was first studied by Tracy and Widom in \cite{TW1} and \cite{TW2}. It 
is now known as the Tracy-Widom distribution.  For the results about local universality in non-perturbed Wigner 
matrices we refer the reader to \cite{J}, \cite{EKYY}, \cite{E}, \cite{TV1}, \cite{TV}, \cite{KY}, \cite{S}, and references therein.
In the case of random $\mb A_N$ the 
universality of the distribution of ``sticking'' eigenvalues was proved by Benaych-Georges, Guionnet, and Maida in \cite{BGM}.

The first step in proving Theorem \ref{thm:caseB} is to use Proposition \ref{Lemma13} to associate the fluctuations of the eigenvalues outside of the 
support of the semicircle law with quadratic forms of the resolvent of the unperturbed Wigner matrix, given by
\[ \mb R_N(z) := (z \mb I_N - \mb X_N)^{-1} .\]  Let us denote by
\begin{equation}
\label{scalarprod}
\langle u, v \rangle:=\sum_{i=1}^N \overline{u_i} \* v_i,
\end{equation}
the standard scalar product in $\C^N$ and by 
\[ \| u \| = \sqrt{\langle u, u \rangle} \]
the induced norm on $\C^N$. 

\begin{proposition}[Proposition 1.1 in \cite{PRS}] 
\label{Lemma13}
Let $\theta_j$ be an eigenvalue of $\mb A_N$ that is greater in magnitude than $\sigma$ and let $\mb u_N^1, \ldots, \mb u_N^{k_j}$ be an orthonormal 
set of eigenvectors of $\mb A_N$ associated with $\theta_j$. 
Let $\mb \Xi^{j}_N $ be the $k_j \times k_j $ matrix with entries
\begin{equation}
\label{thetamatrica}
\Xi^{j}_{lp} :=\sqrt{N}\* \left(\langle \mb u_N^{l}, \mb R_N(\rho_j) \mb u_N^{p}\rangle - g_{\sigma}(\rho_j) \*\delta_{lp}\right)=
\sqrt{N}\* \left(\langle \mb u_N^{l}, \mb R_N(\rho_j) \mb u_N^{p}\rangle - \frac{1}{\theta_j}\*\delta_{lp}\right).
\end{equation}

Let $y_1 \geq \ldots \geq y_{k_j}$ be the ordered eigenvalues of the matrix $\mb \Xi^{j}_N.$
Then
\begin{equation}
\label{uzheuzhe}
\sqrt{N}\*( \lambda_{k_1+\ldots +k_{j-1}+i} -\rho_j )+ \frac{1}{g_{\sigma}'(\rho_j)} \* y_i \to 0,
\ i=1, \ldots, k_j,
\end{equation}
in probability.
\end{proposition}

\begin{remark}
A simple computation gives
\begin{equation}
g_{\sigma}(\rho_j) = \frac{1}{\theta_j}, ~~~ -\frac{1}{g_{\sigma}'(\rho_j)}= \theta_j^2 - \sigma^2.
\end{equation}
\end{remark}

In the second step of proving Theorem \ref{thm:caseB}, we truncate and remove the diagonal terms, defining  
\begin{equation}
\label{benfica}
(\widehat{\mb W}_N)_{ij} :=
(\mb W_N)_{ij} (1-\delta_{ij})\indicator{(\mb W_N)_{ij}\leq N^{1/4}} -\E (\mb W_N)_{ij} (1-\delta_{ij})\indicator{(\mb W_N)_{ij}\leq N^{1/4}},
\end{equation}
$\widehat{\mb X}_N:=\frac{1}{\sqrt{N}} \widehat{\mb W}_N$, and $\mb{ \widehat{R}}_N(z):=(z\*\mb I_N-\widehat{\mb X}_N )^{-1} $ for $z \in \C. $ 
Our goal is to compute the centering, $\E \langle \mb u_N^{l},  \mb{ \widehat{R}}_N(z) \mb u_N^{p} \rangle$, which
in general, is 
dependent on the form of the vectors.  Define
\begin{equation}
\label{wM3}
(\mb{ \widehat{M}}_3)_{ij} := \widehat{\mu}_{3,ij}(1-\delta_{ij}),
\end{equation}
where $\widehat{\mu}_{3,ij} := \E[|\widehat{W}_{ij}|^2\widehat{W}_{ij}]$.

In the appendix (Section \ref{RODT}) we show that the difference
\[\sqrt{N}(\langle \mb u_N,  \mb{ \widehat{R}}_N(z) \mb v_N \rangle - \langle \mb u_N,  \mb{ R}_N(z) \mb v_N \rangle) \]
goes to zero in probability. It is therefore sufficient to use $\E[\langle \mb u_N,  \mb{ \widehat{R}}_N(z) \mb v_N \rangle]$ as the centering for $\langle \mb u_N,  \mb{ R}_N(z) \mb v_N \rangle$ when proving distributional conference (see proof of Theorem \ref{thm:caseB}).

\begin{theorem}
\label{centering}

Let $\mb X_N=\frac{1}{\sqrt{N}} \mb W_N$ be a random real symmetric (Hermitian) Wigner matrix
defined in (\ref{offdiagreal}-\ref{diagreal2}) (respectively (\ref{offdiagherm1}-\ref{diagherm})) and $\widehat{\mb X}_N$, $\widehat{\mb R}_N$ are the
matrices defined after (\ref{benfica}).
Let $\mb u_N, \mb v_N$ be $N$-dimensional unit vectors. 

 \begin{align}
 \label{limitexp}
 \sqrt{N}  \left[ \E \langle \mb u_N,  \mb{ \widehat{R}}_N(z) \mb v_N \rangle 
- g_\sigma(z) \*\langle \mb u_N, \mb v_N \rangle \right] =
\frac{1}{N}g^4_\sigma(z) \langle \mb u_N, \mb{ \widehat{M}}_3 \mb v_N \rangle 
+ O\left(\frac{P_{12}(|\Im(z)^{-1})\*(|z|+1)}{\sqrt{N}}\right), 
 \end{align}
uniformly on $\C \setminus \R,$ 
where 
$P_{12}$ is a polynomial of degree $12$ with fixed positive coefficients.

\end{theorem}
\begin{remark}
\label{zenith}
Since $\|\mb{ \widehat{M}}_3 \| =O(N), $ we have $\frac{1}{N} \* \langle \mb u_N, \mb{ \widehat{M}}_3 \mb v_N \rangle =O(1).$
Furthermore, if $\|\mb u_N\|_1$ or  $ \|\mb v_N\|_1$ is $o(\sqrt{N})$ then  $\frac{1}{N} \* \langle \mb u_N, \mb{ \widehat{M}}_3 \mb v_N 
\rangle=o(1)$.
In addition, since \\
$\widehat{\mu}_{3,ij}-\mu_{3,ij}=O(N^{-1/4})$ uniformly in $i,j,$  it follows that
\begin{equation}
\label{mnogo}
\frac{1}{N} \* \langle \mb u_N, \mb{ \widehat{M}}_3 \mb v_N \rangle- \frac{1}{N} \* \langle \mb u_N, \mb{M}_3 \mb v_N \rangle=O(N^{-1/4}).
\end{equation}
\end{remark}
The proof of Theorem \ref{centering} is in Section \ref{sec:exp}. 

In the final step we compute the joint asymptotic distribution of the matrix entries of $\mb \Xi^{j}_N $. Before stating the theorem we introduce the notation 
\begin{equation}
\label{picov}
\Pi(z_1,z_2) := \left(-g_\sigma(z_1) g_\sigma(z_2) + \frac{g_\sigma(z_1) g_\sigma(z_2)}{1-\sigma^2 g_\sigma(z_1) g_\sigma(z_2)  }\right)
\end{equation}

\begin{theorem}
\label{fluctuations}
Let $\mb X_N=\frac{1}{\sqrt{N}} \mb W_N$ be a random real symmetric (Hermitian) Wigner matrix
defined in (\ref{offdiagreal}-\ref{diagreal2}) (respectively (\ref{offdiagherm1}-\ref{diagherm})).
Let $\mb u_N^{1}, \ldots, \mb u_N^{m}$  be a sequence of mutually orthogonal $N$ dimensional unit vectors such that $\|\mb u_N^{l}\|_{\infty} \to 0$ 
for $l = 1,\ldots, m$ as $N \to \infty$. Let $\mb G_N(z)$ be the $m\times m$ matrix defined by:
\begin{equation}
\label{spartak}
\left(\mb G_N(z)\right)_{lp} :=  \sqrt{N}\left(\langle \mb u_N^{l}, \mb{ \widehat R}_N(z) \mb u_N^{p} \rangle - \E
\langle \mb u_N^{l}, \mb{ \widehat R}_N(z) \mb u_N^{p} \rangle \right).  
\end{equation}
The matrix valued function $\mb G_N(z)$ converges in finite dimensional distributions to the $m \times m$ matrix-valued random field, $\mb \Gamma(z)$ 
with independent, centered, Gaussian entries with covariance given by:
\begin{align*}
\E[\Re(\Gamma_{lp}(z_1))\Re(\Gamma_{lp}(z_2))] &=  \frac{\rho}{4}(\Pi(z_{1},z_{2})  + \Pi(\overline z_{1}, \overline z_{2}) ) + \frac{1}{4}(\Pi(z_{1} ,
\overline z_{2})  + \Pi(\overline z_{1}, z_{2}) )\\
&+\frac{\delta_{lp}}{4}(\rho(\Pi(z_{1},\overline z_{2})  + \Pi(\overline z_{1},  z_{2}) ) + \Pi(z_{1} , z_{2})  + 
\Pi(\overline z_{1},\overline z_{2}) )\\
\E[\Re(\Gamma_{lp}(z_1))\Im(\Gamma_{lp}(z_2))] &= \frac{\rho}{4i}(\Pi(z_1,z_2)  - \Pi(\overline z_1, \overline z_2) ) + \frac{1}{4}(-\Pi(z_1 ,
\overline z_2)  + \Pi(\overline z_1, z_2) )\\
&+\frac{\delta_{lp}}{4i}(\rho(-\Pi(z_1,\overline z_2)  + \Pi(\overline z_1,  z_2) ) + \Pi(z_1 , z_2)  - \Pi(\overline z_1,\overline z_2) )\\
\E[\Im(\Gamma_{lp}(z_1))\Im(\Gamma_{lp}(z_2))] &=  \frac{-\rho}{4}(\Pi(z_1,z_2)  + \Pi(\overline z_1, \overline z_2) ) + \frac{1}{4}(-\Pi(z_1 ,
\overline z_2)  - \Pi(\overline z_1, z_2) )\\
&+\frac{-\delta_{lp}}{4}( \rho(-\Pi(z_1,\overline z_2)  - \Pi(\overline z_1,  z_2) ) + \Pi(z_1 , z_2)  + \Pi(\overline z_1,\overline z_2) )
\end{align*}
for $l \leq p$
and $\Gamma_{lp}(z) = \overline{\Gamma_{pl}(\overline{z})}$ for $l > p$. Where $\rho = 1$ if $\mb W_N$ is real symmetric and $\rho = 0$ if $\mb W_N$ is 
Hermitian.
\end{theorem}
\begin{remark}
\label{remark:1vector}
The result of the above theorem about Gaussian fluctuation of $\left(\mb G_N(z)\right)_{lp}$ can be extended to the case when either
$\|\mb u_N^{l}\|_{\infty} \to 0$ or $\|\mb u_N^{p}\|_{\infty} \to 0.$ See Lemma \ref{limits}.
\end{remark}

For the recent results on the fluctuations of the matrix entries of regular functions of Wigner matrices, we refer the reader to
\cite{LP1}, \cite{PRS1}, \cite{ORS}, \cite{LP2}, \cite{BP}, and \cite{L}.
Theorem \ref{fluctuations} was proved in \cite{BP} in the one vector case, assuming 
that the third moment of the matrix entries vanishes.
We outline the proof of Theorem \ref{fluctuations}  
in Section \ref{flucsec}, showing that the arguments of \cite{BP} can be adapted. 

Finally, we extend the result of Theorem \ref{fluctuations} to a sufficiently large class of regular test functions.
Consider the space $\mathcal{H}_s$ consisting of the functions $f: \R \to \R $ that satisfy
\begin{equation}
\label{sobolev}
\|f\|^2_s:=\int_{\R} (1+|k|)^{2s} \* |\hat{f}(k)|^2 \* dk <\infty,
\end{equation}
where $\hat{f}(k)$ is the Fourier transform 
$$\hat{f}(k)=\frac{1}{\sqrt{2\pi}}\* \int_{\R} e^{-ikx}\*f(x)\*dx.$$

\begin{theorem}
\label{regfunctions}
Let $\mb X_N=\frac{1}{\sqrt{N}} \mb W_N$ be a random real symmetric (Hermitian) Wigner matrix
defined in (\ref{offdiagreal}-\ref{diagreal2}) (respectively (\ref{offdiagherm1}-\ref{diagherm})).
Let $\mb u_N^{1}, \ldots, \mb u_N^{m}$  be a sequence of mutually orthogonal $N$ dimensional unit vectors such that $\|\mb u_N^{l}\|_{\infty} \to 0$ 
for $l = 1,\ldots, m$ as $N \to \infty$.  Finally, let $f: \R \to \R $ belong to $\mathcal{H}_s$ for some $s>4.$
Denote
\begin{equation}
\label{Ylp}
Y_{N,lp}(f):=\sqrt{N}\* \left(\langle \mb u_N^{l}, f(\mb X_N) \mb u_N^{p}\rangle - \E \langle \mb u_N^{l}, f(\mb{ \widehat X}_N) \mb u_N^{p}\rangle \right), 
\end{equation}
$1\leq l,p\leq k.$

Then the joint distribution of $\{Y_{N,lp}, \ 1\leq l\leq p\leq m \}, $ converges as $N \to \infty$ to the distribution of independent centered normal 
random variables with the variance
\begin{equation}
\label{variancelp}
\frac{1+\delta_{lp}}{2\*\beta}\* \int_{-2\sigma}^{2\sigma} \int_{-2\sigma}^{2\sigma} (f(x)-f(y))^2 \* \frac{1}{4\pi^2\sigma^4}\* \sqrt{4\sigma^2-x^2}\*
\sqrt{4\sigma^2-y^2}\* dx\*dy,
\end{equation}
where $\beta=1$ in the real symmetric case and $\beta=2$ in the Hermitian case.  

In addition, if $f$ has a sufficiently large number of derivatives ($13$ is enough) and compact support,
\begin{align}
\label{matozh1}
& \E \langle \mb u_N^{l}, f(\widehat{\mb X}_N) \mb u_N^{p}\rangle = \delta_{lp}  \int_{-2\sigma}^{2\sigma} f(x) \* \frac{1}{2\pi\sigma^2} \* \sqrt{4\sigma^2-x^2}\*dx\*
\\
\label{matozh2}
& + N^{-3/2}\* \langle \mb u_N^{l},  \mb M_3 \mb u_N^{p} \rangle  \int_{-2\sigma}^{2\sigma} f(x)\*(2\*x\*\sigma^{-4} + x^3\*\sigma^{-6}) \* \frac{1}{2\pi\sigma^2} \* \sqrt{4\sigma^2-x^2}\*dx \*   
+ O(N^{-1}).
\end{align} 
\end{theorem}

\begin{remark}
One can extend (\ref{matozh1}-\ref{matozh2}) to $f$ satisfying $\|f\|_{13,1,+}<\infty, $ where
\begin{equation}
\label{fn11}
\|f\|_{n,1,+}:=\max \left( \int_{-\infty}^{+\infty} (|x|+1)\*\left|\frac{d^l f}{dx^l}(x)\right| \* dx, \ 0\leq l\leq n\right).
\end{equation}
\end{remark}

\begin{remark}
Under the additional strong technical assumptions that the matrix entries of $ \mb W_N $ are i.i.d. and
$\log \E \exp(t\*(\mb W_N)_{11}) $ and $\E \exp(t\*|(\mb W_N)_{11}|)$ are entire functions, the first part of the result of Theorem \ref{regfunctions}
follows from Theorem 5.1 in \cite{L}.
\end{remark}

In Section 2 we compute the centering for 
$\langle \mb u_N, \mb{ \widehat R}_N(z) \mb v_N \rangle= \mb u_N^{*} \mb{ \widehat R}_N(z) \mb v_N.$ 
The fluctuation results (Theorems \ref{fluctuations}  and \ref{regfunctions})
are established in Section 3. In Section 4 these results are put together to prove Theorem \ref{thm:caseB}. 
Finally Section 5 contains technical arguments for Section 2 as well reductions on the matrix ensemble under consideration.

\section{Expectation of quadratic forms of the resolvent}
\label{sec:exp}
In this section, we prove Theorem \ref{centering}. We will suppress dependence on $N$ and $z$ when possible.

We begin by refining previous estimates on the expectation of the entries of the resolvent. 
In \cite{PRS1} and \cite{ORS}, we proved the following result
\begin{proposition}[Proposition 3.1 \cite{PRS1}]
\label{proposition:prop1}
Let $\mb X_N=\frac{1}{\sqrt{N}} \mb W_N$ be a random real symmetric (Hermitian) Wigner matrix and $\mb R_N(z)=(z\*\mb I_N-\mb X_N)^{-1} $ 
where $z \in \C. $

Then 
\begin{align}
\label{odinnadtsat100}
&  \E R_{ii}(z) =g_{\sigma}(z) + O \left( \frac{1}{|\Im z|^6 \*N}\right),  \\
\label{odinnadtsat101}
& \E R_{ij}(z)=O \left( \frac{1}{|\Im z|^{5} \*N}\right), \  1\leq i\not=j \leq N,\\
\label{odinnadtsat102}
& \V R_{ij}(z) = O \left( \frac{1}{|\Im z|^{6} \*N}\right), \  1\leq i, j\leq N, 
\end{align}
uniformly on bounded subsets of $\C \setminus \R. $

In addition, if 
\[ \sup_{i\not= j,N} \E[ |(\mb W_N)_{ij}|^5 ] \leq \infty,~~~ \sup_{i,N} \E[|(\mb W_N)_{ii}|^3] \leq \infty,  \]
then
\begin{equation}
\label{odinnadtsat103}
\E R_{ij}(z)=O \left( \frac{1}{|\Im z|^{9} \*N^{3/2}}\right), \  1\leq i\not=j \leq N,
\end{equation}
uniformly on bounded subsets of $\C \setminus \R $.
\end{proposition}

Using Lemma \ref{diagrem} and Section 8.3 of \cite{BP} we set the diagonal entries of $\mb W_N$ to $0$ and truncate the off-diagonal elements at 
$N^{1/4} \epsilon _N$ for some $\epsilon_N \to 0$ as $N \to \infty$ without changing the limiting distribution of $\sqrt{N} \mb u_N^{*} \mb R_N(z) \mb v_N$. We note that \cite{BP} concerns the i.i.d. case but our assumption \eqref{lind2} is 
the averaged version of (8.18) from \cite{BP} and is sufficient for the proof. We also note that the truncation will change the second and third moments 
of the entries of $\mb W_N$ but this error is neglectable and will not be mentioned again.
The estimate \eqref{odinnadtsat103} is improved in the following theorem:
\begin{proposition}
\label{prop:expij3}
Let $\mb X_N=\frac{1}{\sqrt{N}} \mb W_N$ be a random real symmetric (Hermitian) Wigner matrix defined in (\ref{offdiagreal}-\ref{diagreal2}) (respectively (\ref{offdiagherm1}-\ref{diagherm})), and $(\widehat{\mb W}_N)_{ij} = (\mb W_N)_{ij} (1-\delta_{ij})\indicator{(\mb W_N)_{ij}\leq N^{1/4}} -\E[ (\mb W_N)_{ij} (1-\delta_{ij}) \indicator{(\mb W_N)_{ij}\leq N^{1/4}}] $  and $\mb{ \widehat{R}}_N(z)=(z\*\mb I_N-\widehat{\mb X}_N )^{-1} $ where $z \in \C. $ Then for $i \not =j$:
\begin{equation}
\label{expij3}
 \E[ \widehat{R}_{ij}(z)] = \frac{\widehat \mu_{3,ij}}{N^{3/2}} g_\sigma^4(z) +O\left( \frac{P_{12}(|\Im(z)|^{-1})}{N^{2}} \right)
 \end{equation}
uniformly on bounded subsets of $\C \setminus \R$. Where $\widehat \mu_{3,ij} = \E[|\widehat  W_{ij}|^2 \widehat W_{ij}]$.
\end{proposition}

Equation \eqref{expij3} establishes Theorem \ref{centering} by computing $\E[\mb u_N^* \mb{  \widehat{R} }_N(z) \mb v_N]$ for unit vectors $\mb u_N$ and $\mb v_N$.

The proof in the real case is below, the complex case follows similarly.

\begin{proof}
For notational convenience we set $\mb X_N= \mb{ \widehat{ X}}_N$ and    $\mb R_N(z) = \mb{ \widehat{ R}}_N(z)$.
The proof of this theorem is similar to the proof of equation \eqref{odinnadtsat103} in \cite{PRS1} but also uses the estimates established in Proposition \ref{proposition:prop1} of \cite{PRS1} to refine the error terms. The following estimates will be useful:
\begin{equation}
\label{solnce}
\sum_k |R_{ik}|^2 \leq \|\mb R_N \|^2 \leq \frac{1}{|\Im z|^2}, \sum_k |R_{ik}| \leq \frac{\sqrt{N}}{|\Im z|}, \ \text{and} \ |R_{pq}| \leq \frac{1}{|\Im z|},
\end{equation}

In dealing with resolvents, we will use the resolvent identity
\begin{equation}
\label{resident}
	(z I - A_2)^{-1} = (z I - A_1)^{-1} - (z I - A_1)^{-1} (A_1 -A_2) (zI - A_2)^{-1} 
\end{equation}
which holds for all $z \in \C$ where $(z I - A_1)$ and $(z I - A_2)$ are invertible.  

In addition, we will use
the decoupling formula (see for example \cite{KKP}): for any real-valued random variable, $\xi$, with $p+2$ finite moments and $\phi$ a 
complex-valued function with $p+1$ continuous and bounded derivatives
\begin{equation}
\label{decouple}
	\E(\xi \varphi(\xi)) = \sum_{a=0}^p \frac{\kappa_{a+1}}{a!} \E(\varphi^{(a)}(\xi)) + \epsilon 
\end{equation}
where $\kappa_a$ are the cumulants of $\xi$ and $\epsilon \leq C \sup_t | \varphi^{(p+1)}(t)| \E(|\xi|^{p+2})$, $C$ depends only on $p$. We will use the notation $\kappa_{a,ij} $ for the $a^{th}$ cumulant of the $(ij)-{th}$ entry of $\mb W_N$.

The resolvent identity \eqref{resident}, $R_{12} = z^{-1} \sum_k X_{1k}R_{k2}$, and decoupling formula \eqref{decouple} give
\begin{align}
\label{master}
z \E[R_{12}] &= \sigma^2 \E[R_{12} \tr_N(\mb R_N) ] + \frac{\sigma^2}{N} \E[ (\mb R_N^2)_{12} ] \\
&+ \frac{1}{2 N^{3/2}}\sum_k \kappa_{3,1k} \left(4 \E[ R_{12} R_{1k} R_{kk}] + 2\E[R_{11} R_{kk} R_{k2}] + 2\E[R_{1k}^2  R_{2k}] \right)+ r_N \nonumber
\end{align}
Where $r_N$ contains the $3\leq a \leq 6$ terms from the decoupling formula \eqref{decouple} and the error from truncating at $p=6$ in the decoupling formula. Additionally, since we have set $X_{ii} =0$ for $1\leq i \leq N$ we need to add and subtract $\frac{2 \sigma^2}{N} \E[R_{11} R_{12}]$. To ease notation, we also add and subtract $\frac{2^{2a-1} \kappa_{a+1,12} }{N^{(a+1)/2} } \E[R_{12} R_{11}^{a}]$. This term allows future summations to be over all $k$. Using \eqref{odinnadtsat102}, \eqref{solnce} and the Cauchy-Schwarz inequality we see these additionally added terms are $O(N^{-2})$. In the future, when using the decoupling formula we will not mention the error from the diagonal terms.

To estimate the error terms we use Lemma 9.2 of \cite{CDFF}, which states:
\[ \E[R_{12}\tr_{N}(\mb R_N)] - \E[R_{12}] \E[\tr_{N}(\mb R_N)]  = O \left(\frac{P_8(|\Im(z)|^{-1})}{N^2}\right).\]

Since the assumptions in Lemma 9.2 are more restrictive than this paper a few words are needed to justify use of this lemma. Its proof follows by applying the resolvent identity, \eqref{resident}, to $R_{12}$ and then apply the decoupling formula, \eqref{decouple}, to $\sum_k \E[R_{1k} X_{k2}] \E[ \tr_N(\mb R_N(z)]$ and to $\sum_k \E[R_{1k} X_{k2} \tr_N(\mb R_N(z)]$. Then $\E[R_{12}\tr_{N}(\mb R_N)] - \E[R_{12}] \E[\tr_{N}(\mb R_N)]$ is estimated in a similar manner to the proof of Lemma \ref{3term}.

Lemma 9.2 assumes the distribution of the matrix entries satisfy a the Poincar\'e inequality \eqref{poin} and are symmetric. The Poincar\'e inequality assumption gives a bound on the variance of the trace and also implies that all moments of the distribution of the matrix entries are finite.
To bound the variance of the trace we can instead use Lemma 2 of \cite{Shc} which states
\[ \V(\tr_N(\mb R_N(z))) = O \left( \frac{P_4(|\Im(z)|^{-1})}{N^2} \right).\]
Even though the result of Lemma 2 in \cite{Shc} was stated under the assumption that both the off-diagonal and the diagonal matrix entries
have finite fourth moment, the actual proof holds under the assumptions (\ref{offdiagreal}-\ref{diagreal2}) in the real symmetric case
and (\ref{offdiagherm1}-\ref{diagherm}) in the Hermitian case \cite{Sh}. 

Because the entries of $\mb W_N$ have been truncated the $a^{th}$ moment for $a>4$ grows no faster than $N^{(a-4)/4}$, this bound combined with estimates \eqref{odinnadtsat101} and \eqref{odinnadtsat102} allow the higher order error terms in the decoupling formula to be bounded. The proof is lengthy but straight forward and very similar to proof of Lemma \ref{3term}.

Since we do not assume the distribution is symmetric, the third moment might not vanish, but the third cumulant terms in the decoupling formula can be estimated exactly the same as the fourth cumulant terms in \cite{CDFF}, after noting that the summation over $k$ will contribute at most a factor of $N^{1/2}$ instead of $N$ as in the fourth cumulant term.


Furthermore, part (ii) of the Theorem 3.2 in \cite{PRS1} with $f(x) = 1/(z-x)^2$ implies
\[\frac{\sigma^2}{N} \E[ (\mb R_N(z)^2)_{12} ] = O \left(\frac{P_8(|\Im(z)|^{-1})}{N^2}\right)  \]
The following lemma will help bound the third cumulant terms of \eqref{master}. It is 
an improvement of \cite{PRS}, where the Cauchy-Schwarz inequality was used to give an error bound $O(N^{-3/2})$.

\begin{lemma}
\label{3term}
\begin{align}
\label{3termeq}
 &\frac{1}{ N^{3/2}} \E[\sum_k \kappa_{3,1k} (R_{12} R_{1k} R_{kk}+ R_{k2} R_{11} R_{kk} +R_{2k} R_{1k}^2)] \\
 &= \frac{1}{ N^{3/2}}  \sum_k \kappa_{3,1k}( \E[ R_{12}]\E[ R_{1k} R_{kk}] +  \E[ R_{k2}]  \E[ R_{11} R_{kk}] +\E[ R_{2k}] \E[ R_{1k}^2] ) + 
O\left( \frac{P_9(|\Im(z)|^{-1})}{N^2} \right), \nonumber
\end{align}
where $\kappa_{3,ij}$  denotes the third cumulant of the $(ij)$-th entry of $\mb W_N.$
\end{lemma}
The proof is in the appendix.
Applying \eqref{odinnadtsat103} to the right side of the equation in Lemma \ref{3term} and using uniform boundedness of the $\kappa_{3,1k}$'s we obtain

\begin{align*}
&\frac{1}{2 N^{3/2}}\sum_k \kappa_{3,1k} \left(4 \E[ R_{12} R_{1k} R_{kk}] + 2\E[R_{11} R_{kk} R_{k2}] + 2\E[R_{1k}^2  R_{2k}] \right) \\
&= \frac{\kappa_{3,12} }{ N^{3/2}} \E[ R_{22}] \E[ R_{11} R_{22}] + O\left( \frac{P_{11}(|\Im(z)|^{-1})}{N^2} \right).
\end{align*}
This concludes the estimates for the third cumulant terms of \eqref{master}.
The fourth cumulant terms in \eqref{master} are:
\[ \frac{1}{3! N^2}\E\left[\sum_{k} \kappa_{4,1k}( 18R_{11} R_{1k} R_{kk} R_{k2} + 6 R_{11}(R_{kk})^2 R_{12} + 18 (R_{1k})^2 R_{kk} R_{12} + 
6 (R_{1k})^3 R_{k2}) \right] \]
Each of these terms except the second one are shown to be $O(P_4(|\Im(z)|^{-1}) N^{-2})$ in \cite{PRS}. The arguments in Lemma \ref{3term} and estimate 
\eqref{odinnadtsat101} gives the bound
\begin{align*}
\frac{1}{N^2}\sum_{k} \E[  R_{11}(R_{kk})^2 R_{12}] &=\frac{1}{N^2} \sum_{k}\E[  R_{11}(R_{kk})^2] \E[ R_{12}] + 
O\left(\frac{P_{10}(|\Im(z)|^{-1} )}{N^{2} }\right) \\
&= O\left(\frac{P_{10}(|\Im(z)|^{-1} )}{N^{2} }\right)
\end{align*}

Similarly, using that $\E[|W_{ij}|^{4+a}] \leq N^{a/4} \E[|W_{ij}|^4]$ and a lengthy but straightforward calculation the fifth cumulant terms are bounded by 
$O\left(\frac{P_{11}(|\Im(z)|^{-1} )}{N^{9/4} }\right)$. 
 The sixth cumulant terms can be bounded using the Cauchy-Schwarz inequality. For example:
\begin{align*}
\frac{N^{1/2}}{N^{3}} \sum_k \E[R_{kk}^3 R_{11}^2 R_{12}] &= \frac{N^{1/2}}{N^{3}} 
\sum_k \E[R_{kk}^3 R_{11}^2] \E[ R_{12}] + O\left(\frac{|\Im(z)|^{-8}}{N^2} \right) \\
&= O\left(\frac{P_{10}(|\Im(z)|^{-1})}{N^2} \right)
\end{align*}

Finally, the seventh cumulant term is $O (P_7(|\Im(z)|^{-1}) N^{-9/4})$ and the truncation term is the sum of $N$ terms each bounded by 
$O(P_8|\Im(z)|^{-1}) N^{-3}).$

Now we continue the study of the leading order terms of Eq. \eqref{master}.
\begin{align*}
z \E[R_{12}(z)] = \sigma^2 \E[R_{12}(z)] \E[ \tr_N(\mb R_N(z))] + \frac{\kappa_{3,12}}{ N^{3/2}} \E[R_{22}] \E[R_{11} R_{22}] ] + 
O\left( \frac{P_{11}(|\Im(z)|^{-1})}{N^{2}} \right)
\end{align*}

Then by the Cauchy-Schwarz inequality and \eqref{odinnadtsat102}
\begin{align*}
 \frac{\kappa_{3,12}}{ N^{3/2}} \E[ R_{11} R_{22}] &= \frac{\kappa_{3,12}}{ N^{3/2}} \E[ R_{11}]  \E[R_{22}] + 
O\left(\frac{P_6(|\Im(z)|^{-1})}{N^{5/2}} \right) 
 =\frac{\kappa_{3,12}}{ N^{3/2}}\E[ \tr_N(\mb R_N(z))]^2 + O\left(\frac{P_6(|\Im(z)|^{-1})}{N^{5/2}} \right)  \\
\end{align*}

Therefore,
\begin{align*}
\left(z - \sigma^2  \E[ \tr_N(\mb R_N(z))]  \right) \E[R_{12}(z)] =  \frac{\kappa_{3,12}}{ N^{3/2}} \E[ \tr_N(\mb R_N(z))]^3 + 
O\left( \frac{P_{11}(|\Im(z)|^{-1})}{N^{2}} \right)
\end{align*}
This implies (see e.g. \cite{PRS})  that
\begin{align*}
\E[R_{12}(z)] =  \frac{\kappa_{3,12}}{ N^{3/2}} g_\sigma^{4}(z) +O\left( \frac{P_{12}(|\Im(z)|^{-1})}{N^{2}} \right)
\end{align*}
The proof is completed.
\end{proof}

Now we use Proposition \ref{prop:expij3} to prove Theorem \ref{centering}.

\begin{proof}[Proof of Theorem \ref{centering}]
Let $\mb u_N$ and $\mb v_N$ be unit vectors in $\C^N$. Using \eqref{odinnadtsat100} and \eqref{expij3},

\begin{align*}
\E[\mb u_N^{*} \mb R_N(z) \mb v_N] &=  \sum_i \E[R_{ii}(z)] \overline u_i v_i + \sum_{\substack{i,j=1 \\i\not=j}}^N   
\E[R_{ij}(z)]\overline u_i  v_j\\
&=g_\sigma(z) \delta_{lp}  + \frac{g^4_\sigma(z)}{N^{3/2}} \mb u_N^{*} \mb M_3(z) \mb v_N + O(P_{12}(|\Im(z)^{-1}) N^{-1})
\end{align*}

Furthermore, 
if $\|\mb u_N\|_1 = o(\sqrt{N})$ or $\|\mb v_N\|_1 = o(\sqrt{N})$ then
\begin{align*}
\left|\frac{1}{N^{3/2}} \mb u_N^{*} \mb M_3(z) \mb v_N \right| &\leq \frac{\max_{i\neq j} |\mu_{3,ij}|}{N^{3/2}} 
\sum_{i \not= j} |\overline u_{i}  v_{j} |
\leq \frac{\max_{i\not=j} |\mu_{3,ij}| \|\mb u_N\|_1\|\mb v_N\|_1}{N^{3/2}}=o(N^{-1/2})
\end{align*}
\end{proof}

\begin{remark}
If for all $1\leq i,j \leq N$, $\mu_{3,ij} = \mu_3 \in \R$ then 
\begin{align*}
\frac{1}{N^{3/2}} \mb u_N^{*} \mb M_3(z) \mb v_N &= \frac{\mu_3}{N^{3/2}} \left( \sum_{i=1}^{N} \overline u_i \sum_{j=1}^N  v_j -  
\sum_i \overline u_i  v_i  \right) \\
&= \frac{\mu_3}{N^{1/2}} (\mb u_N^{*} \mathbf{1}_N )(\mathbf{1}_N^* \mb v_N) - o(N^{-1/2})
\end{align*}
Where $\mathbf{1}_N = \frac{1}{\sqrt{N}}(1,\ldots, 1)$.
\end{remark}


\section{Fluctuations of bilinear forms}
\label{flucsec}
This section is devoted to the proofs of Theorems \ref{fluctuations}  and \ref{regfunctions}. As in Section 2, for notation convenience we set $\mb X_N= \mb{ \widehat{ X}}_N$ and    $\mb R_N(z) = \mb{ \widehat{ R}}_N(z)$.
Recall that Proposition \ref{Lemma13} implies that Theorem \ref{fluctuations} gives the fluctuations of the outlying eigenvalues. Our argument 
follows the proof of Theorem 3.1 of \cite{BP}. We outline the proof, showing that the theorem holds in the more general case. 

In order to compute the limiting distribution of quadratic forms it is useful to represent sums as Riemann sums of certain 
integrals. In order to effectively make this representation it is useful to permute the elements of the unit vectors. Permuting the entries of unit vectors is equivalent to conjugating our Wigner matrix by a permutation matrix. The matrix obtained from $\mb W_N$ by conjugation by permutation matrices is again a Wigner random matrix satisfying the same hypotheses as $\mb W_N.$
In order to show a good permutation of the unit vectors exist we show that Lemma 3.1 of \cite{BP} can be extended to the case of multiple 
bilinear forms (see Corollary \ref{SteinitzC}).

\begin{theorem}[\cite{Steinitz} see Theorem 1 of \cite{GS}]
\label{Steinitz}
Let $
\{ v_i \}_{i=1}^N$ be a finite family of vectors in $\R^m$ of size $N$. The elements of $v_i$ are denoted $v_{i}^l$ for $1 \leq l \leq m$. Assume 
that $v_i^l \leq c$ for all $i,l$ and $\sum_{i=1}^N v_i =0$. Then there exist a permutation $\pi \in S_N$ and some universal constant $K_m$ 
depending only on $m$, such that
\[ \left\| \sum_{i=1}^{\lfloor N t \rfloor} v_{\pi_i} \right\|_{\infty} \leq  c K_m \]
for all $0\leq t \leq 1,$  where $\|v\|_{\infty}=\max_{1\leq l\leq m} |v^l|$ for $v=(v^1, v^2, \ldots, v^m).$
\end{theorem}


We also have the following corollary.
\begin{corollary} 
\label{SteinitzC}
Let $\mb u^1, \ldots, \mb u^k$ be a set of orthonormal vectors in $\C^N$. There exist a constant $C$, depending only on $k$, and a permutation $\pi \in S_N$ such that
\begin{align*}
& \left| \sum_{i=1}^{\lfloor t N \rfloor} u_{\pi_i}^l \overline u_{\pi_i}^p - \frac{1}{N} \delta_{lp} \right| \leq C  max_{l,k} |u^{l}_{k}|
 \end{align*}
\end{corollary}

\begin{proof}
For each $N$, consider the family $\{v_i\}_{i=1}^N$ such that $ v_i \in \R^{k^2}$ with entries $ |u_i^l|^2 - \frac{1}{N}$ for $1\leq l \leq k$,  $\Re(u_i^{l} \overline u_i^{p})$ and $\Im(u_i^{l}  \overline u_i^{p})$ for   $1\leq l < p \leq k$. Note that for $1\leq q \leq k^2$ and $1 \leq i \leq N$, $v_i^q \leq max_{l,k} |u^{l}_{k}|$.

Then by Theorem \ref{Steinitz} there exists a permutation such that 
\begin{align*}
& \left| \sum_{k=1}^{\lfloor t N \rfloor} u_{\pi_k}^l \overline u_{\pi_k}^p - \frac{1}{N} \delta_{lp} \right| \leq \left| \sum_{k=1}^{\lfloor t N \rfloor} \Re(u_{\pi_k}^l \overline u_{\pi_k}^p - \frac{1}{N} \delta_{lp}) \right| +\left| \sum_{k=1}^{\lfloor t N \rfloor} \Im(u_{\pi_k}^l \overline u_{\pi_k}^p)  \right| \\
 &\leq  2 \left \| \sum_{k=1}^{\lfloor tN \rfloor} v_{\pi_k} \right  \|_{\infty} \leq 2 K_{k^2}  max_{l,k} |u^{l}_{k}|
 \end{align*}
\end{proof}

Now we show that the arguments of \cite{BP} can be extended to the case of multiple bilinear forms. The first step is to write entries of $\mb G_N(z)$ as a martingale 
difference sequence. After estimating error terms the martingale central limit theorem is used. The bulk of the proof will be devoted to computing the 
covariance of the limiting Gaussian distribution.

\begin{proof} [Proof of Theorem \ref{fluctuations}]
Recall that $G_{pl}(z) = \sqrt{N}(\mb u_N^{p*} \mb R_N(z) \mb u_N^{l} - \E[\mb u_N^{p*} \mb R_N(z) \mb u_N^{l} ])$.
By the symmetry of $\mb G_N(z)$ we prove the limiting joint distribution for $G_{pl}(z)$ for $1\leq p \leq l \leq m$. By the Cram\'{e}r-Wold theorem it 
is sufficient to study arbitrary linear combinations of the real and imaginary parts of the matrix elements of $\mb G_N$. So we consider:

\begin{equation}
\label{CW}
 \sum_{1\leq l\leq p\leq m} \sum_{h=1}^n \Re(a_{p,l,h}) \Re(G_{pl}(z_{p,l,h})) +\Im(a_{p,l,h}) \Im(G_{pl}(z_{p,l,h})) 
\end{equation}
for arbitrary complex numbers $a_{p,l,h}$.
We begin by decomposing $G_{pl}(z) $ into a martingale difference sequence:
\[ G_{pl}(z) = \sqrt{N} \sum_{k=1}^N (\E_{k-1} - \E_{k}) \mb u_N^{p*} \mb R_N(z) \mb u_N^{l} \]
where $\E_{k}$ is conditional expectation given $W_{ij}$ for $k < i , j \leq N$ and $\E_N$ is $\E$.

In order to show that the limiting distribution of \eqref{CW} is Gaussian we use the following central limit theorem for martingale difference 
sequences \cite{Bill}.
\begin{theorem}(Theorem 35.12 in \cite{Bill})
\label{MCLT}
For each $N$,  let $Y_{Nk}$ be a martingale difference sequence with respect to $ \mathcal{F}_{n,1},  \mathcal{F}_{n,2}, \ldots $. Let 
\begin{equation}
 \sigma_{Nk}^2 = \E[ Y_{Nk}^2 | \mathcal{F}_{n,k-1}] .
\end{equation} 
Suppose that as $N$ goes to infinity
\begin{equation}
\label{limvar}
\sum_{k=1}^\infty \sigma^2_{Nk} \to_P \sigma^2 >0 
\end{equation} 
and 
\begin{equation}
\label{lind}
\sum_{k=1}^\infty \E[Y_{Nk}^2 I_{|Y_{Nk}| \geq \epsilon}] \to 0 
\end{equation} 
for each $\epsilon >0$. Then 
\begin{equation} 
\sum_{k=1}^{\infty} Y_{Nk} \to_D  N(0,\sigma^2) 
\end{equation} 
\end{theorem}

Before we continue with the proof we introduce some notation.
We use $\mb e_j$ to denote the $j^{th}$ canonical unit vector and $\mb w_k$ to denote the $k^{th}$ column of the matrix $\mb X_N = \frac{1}{\sqrt{N}} \mb W_N$ and $\mb w_j^{(k)}$ denote the $j^{th}$ column of the matrix $\mb X_N $ with the $k^{th}$ entry set to zero. Similarly, by $\mb R_N^{(k)}$ (resp. $\mb R_N^{(jk)}$) we denote resolvent of $\mb X_N - \mb e_k \mb w_k^* - \mb w_k \mb e_k'$ (resp. $\mb X_N - \mb e_k \mb w_k^* - \mb w_k \mb e_k'- \mb e_j \mb w_j^* - \mb w_j \mb e_j'$).  Finally $u_k^l$ is the $k^{th}$ entry of the vector $\mb u_N^l$. We use $o_P(1)$ for terms that converge to zero in probability as $N\to \infty.$

For the reader's convenience we contrast our notation with that of \cite{BP}. They use $-\mb A^{-1}$ (resp. $-\mb A_k^{-1}$, $-\mb A_{jk}^{-1}$) for $\mb R_N(z)$ (resp. $\mb R^{(k)}_N(z)$, $\mb R^{(jk)}_N(z)$). Their unit vectors are denoted $\mb x_N$ and $\mb y_N$ instead of $\mb u_N^{i}$. The vector $\mb w_j$ with the $k^{th}$ element set to 0 is denoted $\mb w_{jk}$ instead of our $\mb w_j^{(k)}$. Finally the Stieltjes transform of the semi-circle law is denoted $s(z)$ instead of $g_\sigma(z)$.

The first observation to be made is that $(\E_{k-1} - \E_{k}) \mb R_N^{(k)}(z) =0$ so it can be subtracted from the martingale decomposition of $ G_{pl}(z)$. Then the resolvent identity, \eqref{resident}, is used to compare $\mb R_N(z)$ and $\mb R_N^{(k)}(z)$, leading to 
\[  G_{pl}(z) = \sqrt{N} \sum_{k=1}^N (\E_{k-1} - \E_{k}) \mb u_N^{p*} \mb R_N(z)(\mb w_k \mb e_k' + \mb e_k \mb w_k^*) \mb R_N^{(k)} \mb u_N^{l} \]

Then $\mb R_N(z)$ can be further expanded by twice applying the identity (see for example \cite{BP} Eq. 2.3)
\begin{equation}
\label{rank1inv}
(\mb B + \mb u \mb v^* )^{-1} = \mb B^{-1} - \frac{\mb B^{-1} \mb u \mb v^* \mb B^{-1} }{ 1 + \mb v^* \mb B^{-1} \mb u} 
\end{equation}
for any invertible matrix $\mb B$ and vectors $\mb u, \mb v$. With $(\mb u, \mb v) = (\mb e_k, \mb w_k)$ and $\mb B = z-\mb X_N$ and then $(\mb w_k,\mb e_k)$ and $\mb B= z-(\mb X_N- \mb e_k \mb w_k^*)$. We remind the reader that $X_{ii} =0 $ for $1\leq i \leq N$.

This leads to 
\[ G_{pl}(z) = \sqrt{N} \sum_k (\E_{k-1}-\E_{k}) \frac{1}{z - \mb w_k^* \mb R^{(k)}_N \mb w_k}  (u_k^l \mb u_N^{p*} \mb R_N^{(k)} \mb w_k - u_k^l u_k^{p*} + \overline u_k^p \mb w_k^* \mb R_N^{(k)} \mb u_N^l + \mb u_N^{p*} \mb R^{(k)}_N \mb w_k \mb w_k^* \mb R^{(k)}_N \mb u_N^{l} ) \]


Checking the estimates in \cite{BP} from Eq. (4.1) - (4.4) leads to 
\[G_{pl}(z_{p,l,h}) = \sqrt{N} g_\sigma(z_{p,l,h}) 
\sum_{k=1}^N \E_{k-1}[ (\zeta_k^{p,l}(z_{p,l,h}) + \psi_k^{p,l}(z_{p,l,h}) + \eta_k^{p,l}(z_{p,l,h})]  +o_P(1)\]
where:
\[\zeta_k^{p,l}(z) := u_k^l \mb u_N^{p*} \mb R_N^{(k)}(z) \mb w_k \text{ and }\psi_k^{p,l}(z) :=\overline u_k^p  \mb w_k^*  \mb R_N^{(k)}(z) \mb u_N^{l}  \]
\[\eta_k^{p,l}(z) :=  \mb w_k^*  \mb R_N^{(k)}(z)  \mb u_N^l  \mb u_N^{p*}  \mb R_N^{(k)}(z)  \mb w_k - 
\frac{1}{N}(\mb u_N^{p*} \left( \mb R_N^{(k)}(z)\right)^2  \mb u_N^{l} + \frac{\overline u_k^p u_k^l}{z^2})\]

It is easily checked that \eqref{CW} satisfies condition \eqref{lind}. The following lemma considers approximations of the variance terms in 
\eqref{limvar} and 
computes the limits of these approximations. The sketch of its proof will take the majority of the remainder of this section. 
\begin{lemma}
Let $\mb u_N^p, \ \mb u_N^q, \ \mb u_N^r,$ and $\mb u_N^l$ be unit $N$-dimensional vectors such that 
any two vectors from this group are either orthogonal or the same. In addition, let us assume that the 
$\| \|_{\infty}$ norm of at least one vector in each pair $\mb u_N^l, \mb u_N^p$ and  $\mb u_N^q, \mb u_N^r$ goes to zero as $N \to \infty.$
Then the following holds
\label{limits}
\begin{align}
\label{zetazeta}
N \sum_{k=1}^{N} \E_k[   \E_{k-1}[\zeta_k^{p,l}(z_1)]\E_{k-1}[\zeta_k^{q,r}(z_2)] ]= 
\rho \sum_{k=1}^N \frac{ \sigma^2 g_\sigma(z_1) g_\sigma(z_2) u^{l}_k u^{r}_k  \sum_{j> k } 
\overline u^{p}_j \overline u^{q}_j}{1 - \sigma^2 \frac{N-k}{N} g_\sigma(z_1) g_\sigma(z_2)}  + o_P(1)
\end{align}
\begin{align}
\label{psipsi}
 N \sum_{k=1}^{N} \E_k[  \E_{k-1}[\psi_k^{p,l}(z_1)]\E_{k-1}[\psi_k^{q,r}(z_2)] ] = 
\rho \sum_{k=1}^N  \frac{ \sigma^2 g_\sigma(z_1) g_\sigma(z_2)\overline u^{p}_k 
\overline u^{q}_k \sum_{j> k } u^{l}_j u^{r}_j}{1 - \sigma^2 \frac{N-k}{N} g_\sigma(z_1) g_\sigma(z_2)}  + o_P(1)
\end{align}
\begin{align}
\label{psizeta}
N \sum_{k=1}^{N} \E_k[  \E_{k-1}[ \psi_k^{p,l}(z_1)] \E_{k-1}[\zeta_k^{q,r}(z_2)] ]=
\sum_{k=1}^N  \frac{ \sigma^2 g_\sigma(z_1) g_\sigma(z_2)\overline u^{p}_k  u^{r}_k  \sum_{j> k } 
\overline u^{q}_j u^{l}_j}{1 - \sigma^2 \frac{N-k}{N} g_\sigma(z_1) g_\sigma(z_2)}  + o_P(1)
\end{align}
\begin{align}
\label{zetapsi}
N \sum_{k=1}^{N} \E_k[ \E_{k-1}[\zeta_k^{p,l}(z_1)]\E_{k-1}[ \psi_k^{q,r}(z_2)]]=
\sum_{k=1}^N \frac{ \sigma^2 g_\sigma(z_1) g_\sigma(z_2)u^{l}_k \overline u^{q}_k  \sum_{j> k } 
\overline u^{p}_j u^{ r}_j}{1 - \sigma^2 \frac{N-k}{N} g_\sigma(z_1) g_\sigma(z_2)}  + o_P(1)
\end{align}
\begin{align}
\label{etaeta}
&N \sum_{k=1}^{N} \E_k[  \E_{k-1}[\eta_k^{p,l}(z_1)]\E_{k-1}[\eta_k^{q,r}(z_2)] ]= \\
&\frac{\rho}{N}\sum_{k=1}^N \frac{ \sigma^4 g^2_\sigma(z_1) g^2_\sigma(z_2)\sum_{i>k} 
\overline u^{p}_i \overline u^{q}_i \sum_{j>k} u^{l}_j u^{r}_j}{(1 - \frac{N-k}{N}\sigma^2 g_\sigma(z_1) g_\sigma(z_2))^2}
+ \frac{1}{N} \sum_{k=1}^N \frac{ \sigma^4 g^2_\sigma(z_1)g^2_\sigma(z_2)\sum_{i>k} 
\overline u^{p}_i u^{r}_i \sum_{j>k} \overline u^{q}_j u^{l}_j}{(1 - \frac{N-k}{N} \sigma^2 g_\sigma(z_1) g_\sigma(z_2))^2}  + o_P(1) \nonumber
\end{align}
\begin{align}
\label{zetaeta}
&N \sum_{k=1}^{N} \E_k[ \E_{k-1}[\zeta_k^{p,l}(z_1)] \E_{k-1}[\eta_k^{p,l}(z_2)]] +\E_{k-1}[\psi_k^{p,l}(z_1)] \E_{k-1}[\eta_k^{p,l}(z_2)] ] \\
&+ \E_{k-1}[\eta_k^{p,l}(z_1)]\E_{k-1}[\zeta_k^{p,l}(z_2)]  +\E_{k-1}[\eta_k^{p,l}(z_1)]\E_{k-1}[\psi_k^{p,l}(z_2)] = o_P(1),  \nonumber
\end{align}
where $o_P(1)$ stands for terms that converge to zero in probability as $N\to \infty.$

Additionally, as $N\to \infty$
\begin{align}
\label{limitvar}
& \sum_{k=1}^N \frac{2  \sigma^2 g_\sigma(z_1) g_\sigma(z_2) \overline u^{p}_k  u^{r}_k  \sum_{j> k } \overline u^{q}_j u^{l}_j}{1 - \sigma^2 \frac{N-k}{N} g_\sigma(z_1) g_\sigma(z_2)} +  \frac{1}{N} \sum_{k=1}^N \frac{ \sigma^4 g^2_\sigma(z_1)g^2_\sigma(z_2)\sum_{i>k} \overline u^{p}_i u^{r}_i \sum_{j>k} \overline u^{q}_j u^{l}_j }{(1 - \frac{N-k}{N} \sigma^2 g_\sigma(z_1) g_\sigma(z_2))^2}  \nonumber\\
&\to \delta_{pr}\delta_{ql} \left(-1 + \frac{1}{1-\sigma^2 g_\sigma(z_1) g_\sigma(z_2)  }\right).
\end{align}

\end{lemma}

\begin{proof}

We begin by proving \eqref{zetazeta}. 

\begin{align*}
&N \E_k[\E_{k-1}[ u^{l}_k \mb u_N^{p*} \mb R^{(k)}(z_1) \mb w_k] \E_{k-1}[ u^{r}_k \mb u_N^{q*} \mb R^{(k)}(z_2) \mb w_k] ] \\
&= N \E_{k}[( u^{l}_k \mb u_N^{p*} \E_{k-1}[  \mb R^{(k)}(z_1) ] \mb I_k \mb w_k) (\mb w_k' \mb I_k \E_{k-1}[\mb R^{(k)}(z_2)^T] \mb{ \overline u}_N^{q} u^{r}_k)] \\
&= \sigma^2 \rho  u^{l}_k u^{r}_k \mb u_N^{p*} \E_{k-1}[  \mb R^{(k)}(z_1) ] \mb  I_k \E_{k-1}[\mb R^{(k)}(z_2)] \mb{ \overline u}_N^{q} 
\end{align*}
Here $\mb I_k := \sum_{j \geq k} \mb e_j  \mb e_j',$ and $\rho = 1$ if $\mb X_N$ is real and $0$ is $\mb X_N$ is complex.

We then apply the following approximation and its conjugate, from \cite{BP}, to $\mb u_N^{p*} \E_{k-1}[  \mb R^{(k)}(z_1) ] \mb  I_k \E_{k-1}[\mb R^{(k)}(z_2)] \mb{ \overline u}_N^{q} $:

\begin{equation}
\label{approxBai}
 \mb e_j' \mb R_N^{(k)}(z) \mb u_N = g_\sigma(z) (u_j +\mb  w_{j}^{k*} \mb R_N^{(jk)}(z) \mb u_N) + o_P(1) 
 \end{equation}
After this approximation is applied, it is shown in \cite{BP} between Eq. (4.9) and Eq. (4.15) that the cross terms are $o_P(1)$. Their analysis carries over directly leading to 
\begin{align}
\label{firstterm}
&\mb u_N^{p*} \E_{k-1}[  \mb R^{(k)}(z_1) ] \mb  I_k \E_{k-1}[\mb R^{(k)}(z_2)] \mb{ \overline u}_N^{q} \\
&= g_\sigma(z_1) g_\sigma(z_2)   \sum_{j>k} \left( \overline u^{p}_j \overline u^{q}_j + \mb u_N^{p*} \E_{k-1}[  \mb R^{(jk)}(z_1) \mb w_j^{(k)}]   \E_{k-1}[\mb w_j^{(k)*} \mb R^{(jk)}(z_2)] \mb{ \overline u}_N^{q} \right) +o_P(1) \nonumber \\
&= g_\sigma(z_1) g_\sigma(z_2) \sum_{j>k}  \left(\overline u^{p}_j \overline u^{q}_j  + \mb u_N^{p*} \E_{k-1}[  \mb R^{(jk)}(z_1) ]\mb I_k \mb w_j^{(k)}   \mb w_j^{(k)*} \mb I_k \E_{k-1}[ \mb R^{(jk)}(z_2)]\mb{ \overline u}_N^{q} \right) +o_P(1) \nonumber\\
&= g_\sigma(z_1) g_\sigma(z_2) \sum_{j>k} \left(\overline u^{p}_j \overline u^{q}_j + \frac{\sigma^2}{N} \mb u_N^{p*} \E_{k-1}[  \mb R^{(jk)}(z_1) ]\mb  I_k \E_{k-1}[ \mb R^{(jk)}(z_2)] \mb{ \overline u}_N^{q} \right)+o_p(1) \nonumber\\
\end{align}
where on the last line, Lemma 8.2 of \cite{BP} is used to show that $\mb I_k \mb w_j^{(k)}   \mb w_j^{(k)*}\mb I_k $ can be approximated by $N^{-1} \mb I_k$.
Then from Eq. (4.16) to Eq. (4.23) it shown that $\mb R^{(jk)}(z)$ can be replaced with $\mb R^{(k)}(z)$ by applying $\eqref{rank1inv}$ twice and using the Cauchy-Schwarz inequality on the error terms. Leading to

\begin{align}
&\mb u_N^{p*} \E_{k-1}[  \mb R^{(k)}(z_1) ] \mb  I_k \E_{k-1}[\mb R^{(k)}(z_2)] \mb{ \overline u}_N^{q} \\
=& g_\sigma(z_1) g_\sigma(z_2) \left(\sum_{j>k} \overline u^{p}_j \overline u^{q}_j + \frac{(N-k)\sigma^2}{N} \mb u_N^{p*} \E_{k-1}[  \mb R^{(k)}(z_1) ]\mb  I_k \E_{k-1}[ \mb R^{(k)}(z_2)] \mb{ \overline u}_N^{q} \right)+o_p(1)\nonumber
\end{align}


Solving for $\mb u_N^{p*} \E_{k-1}[  \mb R^{(k)}(z_1) ] \mb  I_k \E_{k-1}[\mb R^{(k)}(z_2)]  \mb{ \overline u}_N^{q} $, multiplying by $u^l_k u^r_k$ and summing over $k$ gives
\begin{equation}
\sum_k u^l_k u^r_k  \mb u_N^{p*} \E_{k-1}[  \mb R^{(k)}(z_1) ] \mb  I_k \E_{k-1}[\mb R^{(k)}(z_2)] \mb{ \overline u}_N^{q}  = \sum_k u^l_k u^r_k \frac{g_\sigma(z_1) g_\sigma(z_2) \sum_{j>k} \overline u^p_j \overline u^q_j}{1 - \frac{N-k}{N} \sigma^2 g_\sigma(z_1) g_\sigma(z_2)} + o_P(1)
\end{equation}

The proofs of \eqref{psipsi}-\eqref{zetapsi} follow similarly.

The proof of \eqref{etaeta} follows \cite{BP} from Eq. (4.24) to Eq. (4.32).

\begin{align*}
&N\E_k[\E_{k-1}[\eta_k^{p,l}(z_1)]\E_{k-1}[\eta_k^{q,r}(z_2)] ] \\
&=N\E_{k}[ \E_{k-1}[ \mb w_k^* \mb R^{(k)}(z_1)  \mb u_N^{l} \mb  u_N^{p*} \mb R^{(k)}(z_1) \mb w_k - \frac{1}{N}( \mb u_N^{p*} \mb R^{(k)2}(z_1) \mb u_N^{l} + \overline u^{p}_k u^{l}_k/z_1^2)  ] \\
&~~~\E_{k-1}[\mb w_k^* \mb R^{(k)}(z_2)  \mb u_N^{r} \mb  u_N^{q*} \mb R^{(k)}(z_2) \mb w_k - \frac{1}{N}( \mb u_N^{q*} \mb R^{(k)2}(z_2) \mb u_N^{r} + \overline u_k^{q} u_k^{r}/z_2^2)  ]   ]\\
&= \frac{1}{N}\sigma^4 \sum_{i,j>k} \mb e_i' \E_{k-1}[ \mb R^{(k)}(z_1) \mb u_N^{l} \mb  u_N^{p*} \mb R^{(k)}(z_1) ]\mb e_j\mb e_j' \E_{k-1}[ \mb R^{(k)}(z_2)  \mb u_N^{r} \mb  u_N^{q*} \mb R^{(k)}(z_2) ]\mb e_i \\
&+ \frac{\rho}{N}\sigma^4 \sum_{i,j>k} \mb e_i' \E_{k-1}[ \mb R^{(k)}(z_1) \mb u_N^{l} \mb  u_N^{p*} \mb R^{(k)}(z_1) ]\mb e_j\mb e_i' \E_{k-1}[ \mb R^{(k)}(z_2)  \mb u_N^{r} \mb  u_N^{q*} \mb R^{(k)}(z_2) ]\mb e_j \\
&~~~+   \sum_{j>k} \frac{\E[|W_{jk}|^4] - (2 + \rho) \sigma^4 }{N} \mb e_j' \E_{k-1}[ \mb R^{(k)}(z_1)  \mb u_N^{l} \mb  u_N^{p*} \mb R^{(k)}(z_1)] \mb e_j  \mb e_j' \E_{k-1}[ \mb R^{(k)}(z_2)  \mb u_N^{r} \mb  u_N^{q*} \mb R^{(k)}(z_2) ]\mb e_j
\end{align*}
When summed over $k$ the final term is shown to be $o_P(1)$ in \cite{BP}. Using the uniform bound on $\E[|W_{jk}|^4]$ the proof in \cite{BP} also holds in this case. We apply \eqref{approxBai} and its conjugate as before to the first term, the second term follows similarly. As before the cross terms are shown to be error terms in \cite{BP}. 

\begin{align*}
& \sum_{i,j>k}  \mb e_i' \E_{k-1}[ \mb R^{(k)}(z_1) \mb u_N^{l} \mb  u_N^{p*} \mb R^{(k)}(z_1) ]\mb e_j \mb e_j' \E_{k-1}[ \mb R^{(k)}(z_2) \mb u_N^{r} \mb  u_N^{q*}\mb R^{(k)}(z_2) ]\mb e_i \\
&= g_\sigma(z_1) g_\sigma(z_2) \sum_{i,j>k}\left( \mb e_i' \E_{k-1}[ \mb R^{(k)}(z_1) \mb u_N^{l}  (\overline u^{p}_j + \mb u_N^{p*} \mb R^{(jk)}(z_1)  \mb w_j^{(k)}  ) ] \right. \\
&~~~ \left. \E_{k-1}[(u^{r}_j + \mb w_j^{(k)*} \mb R^{(jk)}(z_2) \mb u^{r}_N)  \mb u_N^{q*} \mb R^{(k)}(z_2) ]\mb e_i \right) +o_P(1)\\
&= g_\sigma(z_1) g_\sigma(z_2) \sum_{i,j>k} \left( u^{r}_j \overline u^{p}_j \mb e_i \E_{k-1}[ \mb R^{(k)}(z_1) \mb u_N^{l}  ]  \E_{k-1}[ \mb u_N^{q*}\mb R^{(k)}(z_2) ]\mb e_i \right. \\
&~~~+ \left. \mb e_i' \E_{k-1}[ \mb R^{(k)}(z_1)  \mb u_N^{l} \mb  u_N^{p*}  \mb R^{(jk)}(z_1)   ]\mb I_k \mb w_j^{(k)}     \mb w_j^{(k)*} \mb I_k\E_{k-1}[ \mb R^{(jk)}(z_2)\mb u_N^{r} \mb  u_N^{q*} \mb R^{(k)}(z_2) ]\mb e_i \right)+o_P(1) \\
\end{align*}

Then Lemma 8.2 of \cite{BP} can be used to replace $\mb I_k \mb w_j^{(k)}     \mb w_j^{(k)*} \mb I_k$ with $N^{-1} \mb I_k$ and that $\mb R^{(jk)}(z) $  can be replaced by $\mb R^{(k)}(z)$ by applying $\eqref{rank1inv}$ twice and using the Cauchy-Schwarz inequality on the error terms.  This leads to

\begin{align*}
&= g_\sigma(z_1) g_\sigma(z_2) \sum_{i,j>k} \left( u^{r}_j \overline u^{p}_j \mb e_i' \E_{k-1}[ \mb R^{(k)}(z_1) \mb u_N^{l}  ]  \E_{k-1}[ \mb u_N^{q*}\mb R^{(k)}(z_2) ]\mb e_i \right. \\
&~~~ \left. + \frac{\sigma^2}{N}\mb e_i' \E_{k-1}[ \mb R^{(k)}(z_1)  \mb u_N^{l} \mb  u_N^{p*}  \mb R^{(k)}(z_1)   ]\mb I_k \E_{k-1}[ \mb R^{(k)}(z_2) \mb u_N^{r} \mb  u_N^{q*}  \mb R^{(k)}(z_2) ]\mb e_i \right)+o_P(1) \\
&=g_\sigma(z_1) g_\sigma(z_2) \left( \sum_{j>k}  u^{r}_j \overline u^{p}_j \E_{k-1}[ \mb u_N^{q*}\mb R^{(k)}(z_2) ]\mb I_k \E_{k-1}[ \mb R^{(k)}(z_1) \mb u_N^{l}  ]\right. \\
&~~~+ \left. \frac{N-k}{N}  \sigma^2 \sum_{i,j>k}  \mb e_i' \E_{k-1}[ \mb R^{(k)}(z_1) \mb u_N^{l} \mb  u_N^{p*}  \mb R^{(k)}(z_1)   ]\mb e_j \mb e_j' \E_{k-1}[ \mb R^{(k)}(z_2) \mb u_N^{r} \mb  u_N^{q*}  \mb R^{(k)}(z_2) ]\mb e_i \right) +o_P(1) 
\end{align*}


In the last line, the first term is $g_\sigma(z_1) g_\sigma(z_2)  \sum_{j>k}  u^{r}_j \overline u^{p}_j$ times \eqref{firstterm} and the second term is $g_\sigma(z_1) g_\sigma(z_2)  \sigma^2 (N-k)/N$ times the original expression. Solving for the original expression completes the proof of \eqref{etaeta}.

The remaining terms, \eqref{zetaeta}, contain one $\eta$ and one $\psi$ or $\zeta$. Each term goes to zero in probability, for example:

\begin{align*}
&\left| \sum_k \E_{k}[ \E_{k-1}[ \zeta_k^{p,l}(z_1) ] \E_{k-1}[ \eta_k^{p,l}(z_2) ]]\right| \\
&=\left| \sum_k  \sum_{j>k} \frac{\E[|W_{jk}|^2 W_{jk}] }{\sqrt{N}} \E_{k}[u_k^l \mb u_N^{p*} \E_{k-1}[\mb R_N^{(k)}(z_1) ]\mb e_j \mb e_j' \E_{k-1}[\mb R_N^{(k)}(z_2) \mb u_N^{l} \mb u_N^{p*} \mb R_N^{(k)}(z_2)     ]\mb e_j  ] \right|  \\
&\leq  \frac{C}{\sqrt{N}} \max_{j} |\mb u_N^{p*} \E_{k-1}[\mb R_N^{(k)}(z_1)]\mb e_j| \sum_k u_k^l   \| \E_{k-1}[\mb R_N^{(k)}(z_2) \mb u^{l}_N \| \|\mb u_N^{p*} \mb R_N^{(k)}(z_2)     ]\|
\end{align*}
Which goes to zero because $\max_{j} |\mb u_N^{p*} \E_{k-1}[\mb R_N^{(k)}(z_1)]\mb e_j|=o_P(1)$ as shown in \cite{BP} Eq. (4.31). The other terms are similar.

This completes the proof of the first part of the lemma. We now turn to computing the limit of the sums in \eqref{limitvar}.

We begin with the case where the $\| \|_\infty$ norm of all vectors goes to zero, and assume that the elements of each vector $\{ \mb u^{l} \}_{1\leq l\leq k}$ have been permuted by a permutation given in Corollary \ref{SteinitzC}. Permuting the entries of each vector $\{ \mb u^{l} \}_{1\leq l\leq k}$ is 
equivalent to conjugating $\mb W_N$ by a permutation matrix, which returns a Wigner random matrix.

If $\mb u^q_N \not = \mb u^l_N$ then by orthogonality and Corollary \ref{SteinitzC}, $|\sum_{j>k} \overline u^q_j u^l_j |\leq K_m \max_{i,a} |u^a_i |$. Using the Cauchy-Schwarz inequality shows that the first term in \eqref{limitvar} is bounded by
\begin{align*}
 \sum_k \overline  u^p_k u^r_k \frac{\sum_{j>k} \overline u^q_j  u^l_j}{1 - \frac{N-k}{N} \sigma^2 g_\sigma(z_1) g_\sigma(z_2)}  &\leq K_m \max_{i,a} |u^a_i |  \sum_{k} | \overline u^p_k u^r_k| \leq K_m  \max_{i,a} |u^a_i |  \|\mb u^p_N\| \| \mb u^r_N\| 
\end{align*}
A similar argument works to bound the second term of \eqref{limitvar}.

If  $\mb u^q_N = \mb u^l_N$ but  $\mb u^p_N \not = \mb u^r_N$ then we use the following summation by parts formula to bound the sum.
\[ \sum_{k=1}^{N} g_k f_k = f_1 \sum_{k=1}^{N} g_k + \sum_{k=1}^{N-1} (f_{k+1} - f_k) \sum_{j=k+1}^N g_{j} \]
The formula is applied with
$g_k = \overline u^p_k u^r_k$ and $f_k = \frac{\sum_{j>k} | u^q_j|^2 }{1 - \frac{N-k}{N} \sigma^2 g_\sigma(z_1) g_\sigma(z_2)} $. 
Using the estimate
\begin{align*}
& \left| \frac{\sum_{j>k+1} | u^q_j|^2 }{1 - \frac{N-k-1}{N}\sigma^2 g_\sigma(z_1) g_\sigma(z_2)} - \frac{\sum_{j>k} | u^q_j|^2 }{1 - \frac{N-k}{N}\sigma^2 g_\sigma(z_1) g_\sigma(z_2)} \right| \\
&=  \frac{1}{1 - \frac{N-k-1}{N}\sigma^2 g_\sigma(z_1) g_\sigma(z_2)}  \left|  | u^q_{k}|^2 - \sum_{j>k}| u^q_j|^2 \sigma^2 \frac{ g_\sigma(z_1) g_\sigma(z_2) }{N - (N-k)\sigma^2 g_\sigma(z_1) g_\sigma(z_2)} \right| \\
&\leq C |u^q_{k}|^2 + \frac{C}{N} \sum_{j > k} |u^q_j|^2
\end{align*} 
with the summation by parts formula gives:
\begin{align*}
&\left| \sum_k \overline u^p_k u^r_k \frac{\sum_{j>k} |u^q_j |^2}{1 - \frac{N-k}{N} \sigma^2 g_\sigma(z_1) g_\sigma(z_2)} \right| \\ &\leq \left| \frac{\sum_{j>1} | u^q_j|^2 }{1 - \frac{N-1}{N} \sigma^2 g_\sigma(z_1) g_\sigma(z_2)} K_m \max_{i,a} |u^a_i | \right| + \left| \sum_{k=1}^{N-1} C |u^q_{k}|^2 + \frac{C}{N} \sum_{k=1}^{N-1} \sum_{j > k} |u^q_j|^2\right|  K_m \max_{i,a} |u^a_i |  
 \end{align*}
The right side of the last inequality goes to zero in the limit.
The above arguments can be adapted to the case when the $\| \|_\infty$ norm of at least one of the vectors in each pair $\mb u_N^l, \mb u_N^p$ and  $\mb u_N^q, \mb u_N^r$ goes to zero as $N \to \infty.$

Returning to the case where the $\| \|_\infty$ norm of all the vectors converges to zero. If $p=r$ and $q=l$ then as $N \to \infty$ the Riemann sums converge to the following integrals. 

\begin{align*}
&\sum_{k=1}^N \frac{ \sigma^2 g_\sigma(z_1) g_\sigma(z_2)  |u^{p}_k |^2 \sum_{j> k } |u^{q}_j|^2}{1 - \sigma^2 \frac{N-k}{N} g_\sigma(z_1) g_\sigma(z_2)} 
\to \int_0^1 \frac{ \sigma^2(1-t) g_\sigma(z_1) g_\sigma(z_2)}{ 1 - \sigma^2 (1-t) g_\sigma(z_1) g_\sigma(z_2)  }dt \\
& \frac{1}{N} \sum_{k=1}^N \frac{ \sigma^4 g^2_\sigma(z_1)g^2_\sigma(z_2)\sum_{i>k} |u^{p}_i|^2 \sum_{j>k} |u^{q}_j|^2}{(1 - \frac{N-k}{N} \sigma^2 g_\sigma(z_1) g_\sigma(z_2))^2}  \to \int_0^1 \frac{ ( \sigma^2 g_\sigma(z_1) g_\sigma(z_2)(1-t))^2}{(1- (1-t) \sigma^2 g_\sigma(z_1) g_\sigma(z_2))^2}  \\
\end{align*}
Computing the integrals proves the lemma.

Now we consider the case $|u^{p}_k |^2 = \sum_{i=1}^r \delta_{k a_{i(N)}} b_i$ for some finite $r$ with $\sum b_i = 1$ then this sum can also be computed.

\begin{align*}
& \sigma^2 g_\sigma(z_1) g_\sigma(z_2) \left( \sum_{k=1}^N \frac{ |u^{p}_k |^2 \sum_{j> k } |u^{q}_j|^2}{1 - \sigma^2 \frac{N-k}{N} g_\sigma(z_1) g_\sigma(z_2)} 
+ \sum_{k=1}^N \frac{  |u^{q}_k |^2 \sum_{j> k } |u^{p}_j|^2}{1 - \sigma^2 \frac{N-k}{N} g_\sigma(z_1) g_\sigma(z_2)} \right) \\
&+ \frac{1}{N} \sum_{k=1}^N \frac{ \sigma^4 g^2_\sigma(z_1)g^2_\sigma(z_2)\sum_{i>k} |u^{p}_i|^2 \sum_{j>k} |u^{q}_j|^2}{(1 - \frac{N-k}{N} \sigma^2 g_\sigma(z_1) g_\sigma(z_2))^2} \\
&= \sum_{i=1}^r b_i \left(  \sigma^2 g_\sigma(z_1) g_\sigma(z_2) \left( \frac{1 - a_i/N}{1 -  \sigma^2 g_\sigma(z_1) g_\sigma(z_2) (1- a_i/N)} 
+ \sum_{k=1 }^{a_i} \frac{1/N}{1-  \sigma^2 g_\sigma(z_1) g_\sigma(z_2) (1 - k/N)} \right) \right.\\
&\left. + \frac{1}{N} \sum_{k=1}^{a_i} \frac{ \sigma^4 g_\sigma^2(z_1) g_\sigma^2(z_2) ( 1- k/N)   }{ (1 - (1-k/n) \sigma^2 g_\sigma(z_1) g_\sigma(z_2) )^2   } \right) \\
\end{align*}
The sequence $a_{i(N)}/N$ does not necessarily converge, but it is bounded. Let $A$ be a sub-sequential limit. Along this subsequence, the above term converges to
\begin{align*}
&\to \sum_{i=1}^r b_i  \sigma^2 g_\sigma(z_1) g_\sigma(z_2) \left( \frac{1 - A}{1 -  \sigma^2 g_\sigma(z_1) g_\sigma(z_2) (1-A)} + \int_{0}^A \frac{dt}{1 -  \sigma^2 g_\sigma(z_1) g_\sigma(z_2) (1-t)}    \right) \\
&+ \int_0^A \frac{  \sigma^4 g_\sigma^2(z_1) g_\sigma^2(z_2) (1-t) }{(1- (1-t)  \sigma^2 g_\sigma(z_1) g_\sigma(z_2) )^2} .
 \end{align*}
Computing the integrals shows that the term is independent of $A$.

Furthermore we can consider arbitrary $\mb u^p.$  We begin by permuting the entries so they are non-increasing. Then there exist some $M$ such that for 
all $m>M$, $\mb u^p_m \to 0$. By linearity the above analysis can be applied to the part which goes to zero and the part which does not separately and then combined for the desired result.

\end{proof}

Now we conclude the proof of Theorem \ref{fluctuations} by noting that Lemma \ref{limits} along with the martingale central limit 
(Theorem \ref{MCLT}) implies that centered \eqref{CW} converges in distribution to a Gaussian random variable with variance:

%
%
%
%
%
%
%
%

\begin{align*}
&\sum_{lph h'} \frac{\Re(a_{lph})}{2} \frac{\Re(a_{lph'})}{2} \rho(\Pi(z_{lph},z_{lph'})  + \Pi(\overline z_{lph}, \overline z_{lph'}) ) + (\Pi(z_{lph} ,\overline z_{lph'})  + \Pi(\overline z_{lph}, z_{lph'}) )\\
&+\delta_{lp}(\rho(\Pi(z_{lph},\overline z_{lph'})  + \Pi(\overline z_{lph},  z_{lph'}) ) + (\Pi(z_{lph} , z_{lph'})  + \Pi(\overline z_{lph},\overline z_{lph'}) ))\\
&\frac{\Re(a_{lph})}{2} \frac{\Im(a_{lph'})}{2}  \rho(\Pi(z_{lph},z_{lph'})  - \Pi(\overline z_{lph}, \overline z_{lph'} ) + (-\Pi(z_{lph} ,\overline z_{lph'} )  + \Pi(\overline z_{lph} , z_{lph'} ) )\\
&+\delta_{lp}(\rho(-\Pi(z_{lph},\overline z_{lph'})  + \Pi(\overline z_{lph},  z_{lph'}) ) + (\Pi(z_{lph} , z_{lph'})  - \Pi(\overline z_{lph},\overline z_{lph'}) ))\\
&\frac{\Im(a_{lph})}{2}\frac{\Im(a_{lph'})}{2} \rho(\Pi(z_{lph},z_{lph'})  + \Pi(\overline z_{lph}, \overline z_{lph'}) ) + (-\Pi(z_{lph} ,\overline z_{lph'})  - \Pi(\overline z_{lph}, z_{lph'}) )\\
&+\delta_{lp}( \rho(-\Pi(z_{lph},\overline z_{lph'})  - \Pi(\overline z_{lph},  z_{lph'}) ) + (\Pi(z_{lph} , z_{lph'})  + \Pi(\overline z_{lph},\overline z_{lph'}) ))\\
\end{align*}

Recall that $\Pi(z_1,z_2)$ was defined in \eqref{picov}. The proof of Theorem \ref{fluctuations} is complete.
\end{proof}

Now, we turn our attention to Theorem \ref{regfunctions}.
\begin{proof}[Proof of Theorem \ref{regfunctions}]
Denote 
\begin{equation}
\label{fn1}
\|f\|_{n,1}:=\max \left( \int_{-\infty}^{+\infty} \left|\frac{d^l f}{dx^l}(x)\right| \* dx, \ 0\leq l\leq n\right).
\end{equation}
If $\|f\|_{5,1}<\infty,$ the Gaussian fluctuations for the entries in \eqref{Ylp} follows from Theorem \ref{fluctuations}, and the bound
\begin{equation}
\label{denmark}
\var \left(\langle \mb u_N^{l}, f( \mb X_N) \mb u_N^{p} \rangle \right) \leq Const \frac{\|f\|_{5,1}}{N}
\end{equation}
(equation (1.33) in Theorem 1.6 in \cite{PRS}) by a standard approximation argument (see e.g. the last three paragraphs in Section 4 of \cite{ORS}).
It should be noted that (\ref{denmark}) follows from the bound (2.4) in Proposition 2.1 in  \cite{PRS}, i.e.
\begin{equation}
\label{holland}
\var \left(\langle \mb u_N^{l}, \mb R_N(z) \mb u_N^{p} \rangle \right) = O \left(\frac{P_8(|\Im z|^{-1})}{N}\right)
\end{equation}
by applying Helffer-Sj\"ostrand functional calculus (\cite{HS}, \cite{D}).  To prove Gaussian fluctuation for an arbitrary function
$f\in \mathcal{H}_s$ with $s>4, $ one has to strengthen (\ref{holland}) and use
\begin{equation}
\label{holland1}
\var \left(\langle \mb u_N^{l}, \mb R_N(z) \mb u_N^{p} \rangle \right) = O \left(\frac{(\E \|\mb R_N(z)\|^2) \* P_6(|\Im z|^{-1})}{N}\right)+ 
O \left(\frac{(\E \|\mb R_N(z)\|^{3/2}) \*P_6(|\Im z|^{-1})}{N}\right)
\end{equation}
by repeating the steps of Proposition 3.2 in \cite{ORS}.  One then applies Proposition 2.2 in \cite{ORS} (see also Proposition 2 in \cite{Shch} or
Proposition 1 in \cite{Shc})  to prove that
\begin{equation}
\label{denmark1}
\var \left(\langle \mb u_N^{l}, f( \mb X_N) \mb u_N^{p} \rangle \right)\leq Const_s \frac{\|f\|_{s}}{N}.
\end{equation}
The Gaussian fluctuation then follows by a standard approximation argument as before. 
The estimate of the mathematical expectation of $\langle \mb u_N^{l}, f( \widehat{\mb X}_N) \mb u_N^{p} \rangle$ follows by applying
Theorem \ref{centering} and the Helffer-Sj\"ostrand functional calculus.
Theorem \ref{regfunctions} is proven.
\end{proof}

\section{Proof of the Theorem \ref{thm:caseB}}
\label{mainthm}

\begin{proof}[Proof of Theorem \ref{thm:caseB}]
We begin with $|\theta_j |> \sigma$ an eigenvalue of $\mb A_N$ with multiplicity $k_j$, the orthonormal eigenvectors of $\mb A_N$ corresponding to 
$\theta_j$ are labeled $\mb u_N^{1}, \ldots,\mb  u_N^{k_j}$.
Following Theorem \ref{fluctuations}, $\mb G_N^{j}(z)$ is the $k_j \times k_j$ matrix with entries
\[ (\mb G_N^{j}(z))_{pl} =  \sqrt{N}(\mb u_N^{p*} \mb{ \widehat R}_N(z) \mb u_N^{l} - \E[\mb u_N^{p*} \mb{ \widehat R}_N(z) \mb u_N^{l} ])  \]

By Proposition \ref{Lemma13}, the fluctuations of the eigenvalues can be expressed in terms of 
the fluctuations of the eigenvalues of $\mb \Xi_N^{j}$. Then using the definition of $\mb \Xi_N^{j}$, 
the estimate on $\sqrt{N}((\langle \mb u_N^l , \mb R_N(z) \mb u_N^p \rangle - \langle \mb u_N^l, \mb{ \widehat R}_N(z)\mb u_N^p \rangle )$ and 
Theorem \ref{centering} leads to:
\begin{align*}
\Xi^{j}_{lm} &= \sqrt{N}(\langle \mb u_N^l, \mb R_N(z) \mb u_N^p \rangle  - \frac{1}{\theta_j} \delta_{lp} )\\
&=  \sqrt{N}((\langle \mb u_N^l , \mb R_N(z) \mb u_N^p \rangle - \langle \mb u_N^l, \mb{ \widehat R}_N(z)\mb u_N^p \rangle ) + 
(\langle \mb u_N^l, \mb{ \widehat R}_N(z)\mb u_N^p \rangle  - \E[\langle \mb u_N^l, \mb{ \widehat R}_N(z)\mb u_N^p \rangle ])\\
& + (\E[\langle \mb u_N^l, \mb{ \widehat R}_N(z)\mb u_N^p \rangle ] - \frac{1}{\theta_j} \delta_{lp}))\\
&= \sqrt{N} (\langle \mb u_N^l, \mb{ \widehat R}_N(z)\mb u_N^p \rangle  - \E[\langle \mb u_N^l, \mb{ \widehat R}_N(z)\mb u_N^p \rangle ]) 
+ \sqrt{N} (\E[\langle \mb u_N^l, \mb{ \widehat R}_N(z)\mb u_N^p \rangle ] - \frac{1}{\theta_j} \delta_{lp}) \\
& +\sqrt{N}(\langle \mb u_N^l , \mb R_N(z) \mb u_N^p \rangle - \langle \mb u_N^l, \mb{ \widehat R}_N(z)\mb u_N^p \rangle )\\
&= \mb G^j_N(\rho_j)_{lm} + \frac{1}{\theta^4 N} \mb u_N^{l*} \mb M_3 \mb u_N^{m} + o_P(1),
\end{align*}
where $ \sqrt{N}(\langle \mb u_N^l , \mb R_N(z) \mb u_N^p \rangle - \langle \mb u_N^l, \mb{ \widehat R}_N(z)\mb u_N^p \rangle )=o_P(1) $
follows from Lemma 2.1 of \cite{ORS} and Lemma \ref{diagrem} in the Appendix.
 By Theorem \ref{fluctuations}, 
$\mb G^j_N(z)$ converges weakly in finite dimensional distributions to the $k_j \times k_j$ matrix valued random field $\mb \Gamma(z)$ with independent centered 
entries that are Gaussian. Since the difference between $\mb G^j_N(z)$ and $\mb \Xi^{j}$ converges to a constant in probability, $\mb \Xi^{j}$ converges in finite dimensional distributions to $\mb \Gamma(z)$ plus that constant. All that is left to check is that entries have the announced variance.

First note that:
\begin{align*}
\Pi(\rho_j,\rho_j) =& \left(g(\rho_j)^2 - g'(\rho_j)  \right) =  \left(-\frac{1}{\theta^2} + \frac{1}{\theta^2 -\sigma^2} \right)\\
& \left(-\frac{\theta^2 +\sigma^2-\theta^2}{\theta^2(\theta^2 -\sigma^2)}\right) =  \left(\frac{ \sigma^2}{\theta^2(\theta^2 -\sigma^2)}\right)\\
\end{align*}
So we conclude:
\begin{align*}
\lim_{N\to \infty} \E[\Re(\Gamma_{lp}(z_1))\Re(\Gamma_{lp}(z_2))] &=  \frac{(\delta_{lp}+1)(\rho+1)}{2}\left(\frac{ \sigma^2}{\theta^2(\theta^2 -\sigma^2)}\right)\\
\lim_{N\to \infty} \E[\Re(\Gamma_{lp}(z_1))\Im(\Gamma_{lp}(z_2))] &= 0\\
\lim_{N\to \infty} \E[\Im(\Gamma_{lp}(z_1))\Im(\Gamma_{lp}(z_2))] &=  \frac{(1-\delta_{lp})(1-\rho)}{2}\left(\frac{ \sigma^2}{\theta^2(\theta^2 -\sigma^2)}\right)
\end{align*}

Multiplying by $ (c_{\theta}/g'_\sigma(\rho) )^2= \left( \frac{\theta^2}{\theta^2-\sigma^2}\right)^2  (\theta^2 - \sigma^2)^2$ gives the desired variance.

\end{proof}

\section{Appendix}

\begin{proof}[Proof of Lemma \ref{3term}]

We will prove the first estimate, at the end of the proof we note that the other two estimates can be proved similarly.
We apply the resolvent identity \eqref{resident} and decoupling formula \eqref{decouple} to 
$\sum_{k} \E[R_{12} R_{1k} R_{kk}]$ and $\sum_k\E[R_{12}] \E[ R_{1k} R_{kk}]$. After estimating error terms we study the difference
of $\sum_{k} \E[R_{12} R_{1k} R_{kk}]$ and $\sum_k\E[R_{12}] \E[ R_{1k} R_{kk}].$
\begin{align}
\label{res3cum1}
z \sum_{k} \E[R_{12}R_{1k} R_{kk}] = \E[\sum_{l,k} X_{1l} R_{l2}  R_{1k} R_{kk}]
\end{align}
In the decoupling formula, \eqref{decouple}, the second cumulant term is
\begin{align*}
&=\frac{\sigma^2}{ N} \sum_{k} \E[ ((\mb R_N^2)_{12} +\Tr(\mb R_N) R_{12})  R_{1k} R_{kk} +  (\mb R_N^2)_{2k} R_{11}  R_{kk}+ (\mb R_N^2)_{12} R_{1k} R_{kk} \\
&+ (\mb R_N^2)_{2k} R_{1k} R_{k1} + (\mb R_N^2)_{2k} R_{1k} R_{1k} ]\\
&= \frac{\sigma^2}{ N} \sum_{k} \E[\Tr(\mb R_N) R_{12}  R_{1k} R_{kk}] + O\left(\frac{|\Im(z)|^{-4}}{N^{1/2}}\right)
\end{align*}
We used \eqref{solnce} to estimate the error.

Each of the third cumulant terms will have 3 $l$'s and 3 $k$'s in the matrix subscripts. For example:
\[ \frac{1}{N^{3/2}} \sum_{k,l} \kappa_{3,1l} R_{ll} \E[R_{11} R_{l2} R_{1k} R_{kk}] = \frac{1}{N^{3/2}} \E[\sum_{l} \kappa_{3,1l} R_{ll} R_{11} R_{l2} \sum_{k} R_{1k} R_{kk}].   \] 
Using \eqref{solnce} the absolute value of this term is bounded by $ O\left(\frac{|\Im(z)|^{-5}}{N^{1/2}}\right)$. All the third cumulant terms can be bounded in the same manner except:
\begin{align*} 
& \left|\frac{1}{N^{3/2}} \sum_{k,l} \kappa_{3,1l} R_{ll} R_{12} R_{11} R_{lk} R_{kk} \right| \leq \left|\frac{K}{N^{3/2}} R_{12} R_{11} \mb R_N^{(D)} \mb R_N \mb R_N^{(D)} \right| \\
&\leq  \frac{K}{N^{3/2}} |\Im(z)|^{-2} \|\mb R_N^{(D)}\| \|\mb R_N\| \|\mb R_N^{(D)}\| = O\left( \frac{|\Im(z)|^{-5}}{N^{1/2}} \right),
\end{align*}
where $\mb R^{(D)} $ is the N-dimensional vector with $\mb R^{(D)}_{i} = R_{ii}$. 

The fourth cumulant terms are of the form 
\[\left| \frac{1}{N^2} \sum_{l,k} \kappa_{4,1l}R_{ka} R_{**}R_{**}R_{**}R_{**}R_{**} \right| \leq \frac{K|\Im(z)|^{-5} }{N} \sum_{k} |R_{ka}| = 
O\left(\frac{|\Im(z)|^{-6}}{N^{1/2}}  \right),  \]
where $a\in \{1,2,l\}.$ 
A similar argument works for the fifth cumulant terms, using that $\kappa_{5,1l} \leq const \* N^{1/4}$.

Finally the truncation term is the sum of $N^2$ terms, each bounded by $|\Im(z)|^{-8} N^{-5/2}$.
Applying the above estimates to \eqref{res3cum1} leads to:
\begin{align}
\label{varest12k}
z \sum_{k} \E[R_{12}R_{1k} R_{kk}] = \sigma^2 \sum_{k} \E[\tr_N(\mb R_N) R_{12}  R_{1k} R_{kk}] + O\left(\frac{P_8(|\Im(z)|^{-1})}{N^{1/2}}\right).
\end{align}

Now we apply the resolvent identity \eqref{resident} to the other term of interest.
 \[z \sum_{k} \E[R_{12}] \E[R_{1k} R_{kk}] = \sum_{k} \E[ R_{12}] \E[R_{kk}\delta_{k1} + \sum_{l} X_{1l} R_{lk} R_{kk}].\] 
Before applying the decoupling formula note that by \eqref{odinnadtsat101}
\[ |\E[ R_{12}] \E[ R_{11} ]| \leq \frac{P_6(|\Im(z)|^{-1})}{N} \]
The second cumulant term is:
\begin{align*}
&\frac{\sigma^2}{N} \sum_{l,k} \E[ R_{12}] \E[(R_{l1}R_{lk} +R_{ll}R_{1k} ) R_{kk} + R_{lk}(R_{k1}R_{lk}+R_{kl}R_{1k})  ]\\
&= \frac{\sigma^2}{N} \sum_{k} \E[ R_{12}]\E[ \Tr(\mb R_N)  R_{1k} R_{kk}] + O\left(\frac{|\Im(z)|^{-4}}{N^{1/2}}\right)
\end{align*}
The higher order terms can be bounded as before leading to:
\begin{align}
\label{varest12-k}
z \sum_{k} \E[R_{12}] \E[ R_{1k} R_{kk}] = \sigma^2 \sum_{k} \E[ R_{12}] \E[ \tr_N(\mb R_N) R_{1k} R_{kk}] + O\left(\frac{P_8(|\Im(z)|^{-1})}{N^{1/2}}\right)
\end{align}

Taking the difference of \eqref{varest12k} and \eqref{varest12-k} and subtracting $\sum_k \sigma^2 \E[\tr_N(R)] \E[ R_{1k} R_{kk}(R_{12} -\E[ R_{12}] )] $
from both sides of the equation leads to:
\begin{align*}
&(z-\sigma^2  \E[\tr_N(\mb R_N)] ) \sum_{k} \E[R_{1k} R_{kk}(R_{12} -\E[ R_{12}] )]  \\
&=  \sigma^2 \sum_{k} \E[\tr_N(\mb R_N) R_{12}  R_{1k} R_{kk}] -\sigma^2 \sum_{k} \E[ R_{12}]\E[ \tr_N(\mb R_N) R_{1k} R_{kk}] \\
&~~~ -  \sigma^2 \E[\tr_N(\mb R_N)] \sum_{k} \E[R_{1k} R_{kk}(R_{12} -\E[ R_{12}] )] + O\left(\frac{P_8(|\Im(z)|^{-1})}{N^{1/2}}\right) \\
&=   \sigma^2 \sum_{k}  \E[(\tr_N(\mb R_N)-\E[\tr_N(\mb R_N)]) R_{1k} R_{kk}(R_{12} -\E[ R_{12}] )]  + O\left(\frac{P_8(|\Im(z)|^{-1})}{N^{1/2}}\right) 
\end{align*}
Therefore,
\begin{align*}
&|(z-\sigma^2  \E[\tr_N(\mb R_N)] ) \sum_{k} \E[R_{12}( R_{1k} R_{kk} -\E[  R_{1k} R_{kk}] )]| \\
&\leq   \sigma^2 \sum_{k} |\Im(z)|^{-2} \V(R_{12})^{1/2} \V(\tr_N(\mb R_N))^{1/2} + O\left(\frac{P_8(|\Im(z)|^{-1})}{N^{1/2}}\right) 
\end{align*}
which implies
\begin{equation*}
| \sum_{k} \E[R_{12}( R_{1k} R_{kk} -\E[  R_{1k} R_{kk}] )]|\leq \frac{P_9(|\Im(z)|^{-1})}{N^{1/2}},
\end{equation*}
where we use the bound $ |(z- \sigma^2  \E[\tr_N(\mb R_N)])^{-1}| \leq |\Im(z)|^{-1}$.

The second estimate is proved in the same manner, beginning with the resolvent identity \eqref{resident} $zR_{k2} = (\delta_{k2} + \sum_l R_{kl} X_{l2} ) $.

Similarly, to study $\sum_{k} \E[R_{2k} R_{1k}^2]$  we apply the resolvent identity \eqref{resident} $z \sum_{k}R_{2k}R_{1k}^2  = R_{12}^2 + \sum_{l,k} X_{2l} R_{lk}  R_{1k}^2$ and to study $\sum_{k} \E[R_{2k}] \E[R_{1k}^2]$  we apply $ z \sum_{k} \E[R_{2k}] \E[R_{1k}^2] = \E[ R_{2k}] \E[  \sum_{l,k} X_{1l} R_{lk} R_{1k}] + \E[\frac{1}{z} R_{21}] \E[R_{11}]$ and proceed as in the proof of \eqref{3termeq}

\end{proof}

\subsection{Removal of Diagonal Terms}
\label{RODT}

In this section we truncate the entries of $\mb W_N$ and remove the diagonal elements that are not relevant in the limiting distribution of $\mb u_N^* \mb R_N(z) \mb v_N$. In this procedure we will consider arbitrary unit vectors $\mb u_N$ and $\mb v_N$ in order to change the assumption of five finite moments to the optimal four finite moments in \cite{PRS}. Because conjugating a Wigner matrix by a permutation matrix gives a Wigner matrix we can without loss of generality assume that there exist some finite $m$ such that the entries of $\mb u_N$ and $\mb v_N$ that do not go to zero are in the first $m < \infty$ coordinates for all $N$. In what follows we will set all the diagonal elements $W_{ii}$ for $i>m$ equal to zero. 

We note that the results \cite{PRS} show that the diagonal entries of $\mb W_N$ cannot be removed without effecting the limiting distribution if the $\| \|_\infty$ of $\mb u_N$ and $\mb v_N$ do not go to zero.

We begin by noting that Lemma 2.1 of \cite{ORS} allows us to replace $\mb W_N$ with another Wigner matrix whose first 2 moments match but its entries are bounded by $\epsilon_N \sqrt{N}$ for some $\epsilon_N$ that goes to $0$. For the remainder of the section we assume that we are working with the new matrix.

\begin{lemma}
\label{diagrem}
Let $\mb X_N=\frac{1}{\sqrt{N}} \mb W_N$ be a random real symmetric (Hermitian) Wigner matrix
defined in (\ref{offdiagreal}-\ref{diagreal2}) (respectively (\ref{offdiagherm1}-\ref{diagherm})).
Let $\mb u_N,\mb v_N$ be a sequence of orthogonal $N$ dimensional unit vectors such that $(\mb u_N)_i,(\mb v_N)_i \to 0$ for all $i>m$ and some $m < \infty$.
Let $\text{diag}_m(\mb X_N)$ be the diagonal matrix such that $\text{diag}_m(\mb X_N)_{ii} = (\mb X_N)_{ii}$ for $i>m$ and $0$ otherwise and $\widehat{\mb R}_N(z):= (z \mb I_N - (\mb X_N - \diag_m(\mb X_N)))^{-1}$
 \[\sqrt{N}( \mb u_N^* \mb R_N(z) \mb v_N -  \mb u_N^* \widehat{\mb R}_N(z) \mb v_N)\to_P 0.\]
 \end{lemma}
\begin{proof}
Let $\epsilon > 0$. We begin by noting that by Proposition 2.1 of \cite{ORS} the event $\Omega := \{ \mb X_N\big| \|\mb X_N\| < 2\sigma + \epsilon\}$ has measure going to one as $N \to \infty$. So it is sufficient to prove convergence on this event. 

Then using the resolvent identity 
\begin{align*}
\sqrt{N}( \mb u_N^* \mb R_N(z) \mb v_N -  \mb u_N^* \widehat{\mb R}_N(z) \mb v_N )&= \sqrt{N}( \mb u_N^* \mb R_N(z) \diag_m(\mb X_N) \widehat{\mb R}_N(z) \mb v_N) \\
&= \sum_{i>m} \Big(\mb u_N^* \mb R(z)\Big)_i W_{ii} \left(\widehat{ \mb R}(z) \mb v_N\right)_i 
\end{align*}
Let $\Omega' := \{\mb X_N \big| \|\mb X_N - \diag_m(\mb X_N)\| < 2\sigma + \epsilon + \epsilon_N\}$ and note that $\Omega \subseteq \Omega'$. The set $\Omega'$ is useful because the resolvent $\mb R^{ii}_N(x)$ (defined below) for real $x > 2 \sigma + \epsilon + \epsilon_N$ is bounded by $(x - (2 \sigma + \epsilon + \epsilon_N))^{-1}$.

We now show the above term goes to zero in $L^2$.
\begin{align}
&\E\left[\left| \sum_{i>m} \Big(\mb u_N^{*} \mb R(z)\Big)_i W_{ii} \Big(\widehat{ \mb R}(z) \mb v_N\Big)_i \right|^2 \indicator{\Omega}\right] \leq \E\left[\left| \sum_{i>m} \Big(\mb u_N^{*} \mb R(z)\Big)_i W_{ii} \Big(\widehat{ \mb R}(z) \mb v_N\Big)_i  \right|^2 \indicator{\Omega'} \right] \nonumber \\
\label{diterm}
&= \sum_{i>m} \E\left[\left|  \Big(\mb u_N^{*} \mb R(z)\Big)_i W_{ii} \Big(\widehat{ \mb R}(z) \mb v_N\Big)_i \right|^2\indicator{\Omega'}\right] \\
\label{crossterm}
&+ \sum_{i\not=j>m} \E\left[\left(  \Big(\mb u_N^{*} \mb R(z)\Big)_i W_{ii} \Big(\widehat{ \mb R}(z) \mb v_N\Big)_i \right) 
\left(  \Big(\mb v_N^{*} \widehat{\mb R}(\overline z)\Big)_j W_{jj} \Big( \mb R(\overline z) \mb u_N\Big)_j \right)\indicator{\Omega'}\right]
\end{align}
 
Let $\mb R_N^{(ii)}(z) = (z \mb I_N - (\mb X_N - \mb e_i X_{ii} \mb e_i'  )  )^{-1}$. Applying \eqref{rank1inv}  gives

\begin{align*}
 \mb R_N(z) \mb e_i &=   (z \mb I_N - (\mb X_N - \mb e_i X_{ii} \mb e_i'  ) - \mb e_i X_{ii} \mb e_i'    )^{-1} \mb e_i \\
 &= \mb R_N^{(ii)}(z) \mb e_i (1 - X_{ii} ( \mb e_i' \mb R_N^{(ii)}(z) \mb e_i))^{-1}\\
\end{align*}
Let $\beta_{i} := (1 - X_{ii} ( \mb e_i' \mb R_N^{(ii)}(z) \mb e_i))^{-1}$.

We bound  \eqref{diterm} by:
\begin{align}
 &\sum_{i>m} \E[|  (\mb u_N^{*} \mb R(z))_i W_{ii} (\widehat{ \mb R}(z) \mb v_N)_i|^2\indicator{\Omega'}]  \nonumber\\
 &=  \sum_{i>m} \E[|  (\mb u_N^{*} \mb R^{(ii)}(z))_i W_{ii} (\widehat{ \mb R}(z) \mb v_N)_i  \beta_{i} |^2\indicator{\Omega'}] \nonumber \\
  \label{quin}
 &=\sum_{i>m} \E[W_{ii}^2] \E[|  (\mb u_N^{*} \mb R^{(ii)}(z))_i (\widehat{ \mb R}(z) \mb v_N)_i|^2 \indicator{\Omega'}] \\
  \label{betaerror}
 &- \E[|  (\mb u_N^{*} \mb R^{(ii)}(z))_i W_{ii} (\widehat{ \mb R}(z) \mb v_N)_i|^2 (1 - |\beta_i|^{2}) \indicator{\Omega'}  \ ]
\end{align}
The first part of this equation \eqref{quin} is bounded by
\begin{align*}
&\sum_{i>m} \E[W_{ii}^2] \E[|  (\mb u_N^{*} \mb R^{(ii)}(z))_i (\widehat{ \mb R}(z) \mb v_N)_i|^2 \indicator{\Omega'}]   \\
&\leq \max_{i} \E[W_{ii}^2]  \sum_{i>m}  \left( \sum_{k > m}   \E[| (\mb u_N^{*} \mb R^{(ii)}(z))_i \widehat{ R}_{ik}  v_{k}|^2 \indicator{\Omega'}] + \sum_{ k \leq  m}  \E[| (\mb u_N^{*} \mb R^{(ii)}(z))_i \widehat{ R}_{ik}  v_{k}|^2 \indicator{\Omega'}]  \right) \\
&\leq C \left( \max_{k > m} |v_k|^2  \sum_{k > m}   \E[| ( \mb u^*_N \mb R^{(ii)}(z) \mb I_{m+1} \widehat{\mb  R}(z))_{k} |^2 \indicator{\Omega'}]   + C m   \max_{k\leq m} \E[ \max_{i>m}  |\widehat{ R}_{ik} v_k|^2  \indicator{\Omega'}]     \right)   
\end{align*}
%
%
The first term of this term converges to zero because $v_k \to 0$ for all $k>m$. The second term converges to zero because the $\E[ max_{i\not= k} \widehat R_{ik} \indicator{\Omega'}]$ converges to zero, indeed starting from the following identity for an off-diagonal resolvent entry (see for example \cite{E}) 

\[ \widehat R_{ik} = \widehat R_{ii} \widehat R^{(i)}_{kk}(X_{ik} + \mb w_{i}^* \widehat{\mb  R}^{(ik)} \mb w_{k}) \]

%

Let $\tilde \epsilon > 0$. Using $\E[\mb w_{i}^* \widehat{\mb  R}^{(ik)} \mb w_{k}] =0$ and Markov's inequality gives that
\begin{align*}
\P( | \mb w_{i}^* \widehat{\mb  R}^{(ik)} \mb w_{k}| > \tilde \epsilon) \leq  \frac{\E[|\mb w_{i}^* \widehat{\mb  R}^{(ik)} \mb  w_{k}|^4]}{\tilde\epsilon^4} \leq \frac{\epsilon_N}{N \tilde\epsilon^{4}}
\end{align*}
then using the crude bound on $\widehat R_{ii}$
\begin{align*}
\P( max_{i>m} |\widehat R_{ik}| > \tilde \epsilon) &\leq  \P(max_{i>m} C| X_{ik} +  \mb w_{i}^* \widehat{\mb  R}^{(ik)} \mb  w_{k}|> \tilde \epsilon) \\
&\leq C( \P(max_{i>m} C| X_{ik}| > \tilde \epsilon) + N \P ( C|\mb w_{1}^* \widehat{\mb  R}^{(1k)} \mb w_{k}|> \tilde \epsilon) )\\
&\to 0.
\end{align*}
Along with the fact that $|\widehat R_{ik}(z)|\indicator{\Omega'}$ is bounded shows that $\E[max_{i>m} \widehat R_{ik} \indicator{\Omega'}]$ goes to zero and hence \eqref{diterm} goes to zero as $N \to \infty$.


The second term \eqref{betaerror} converges to zero because:

\begin{align}
\label{betabound}
(1 - \beta_i) \indicator{\Omega'} &= \frac{   X_{ii}  R_{ii}^{(ii)}(z) }{1 - X_{ii}  R_{ii}^{(ii)}(z) }\indicator{\Omega'}
\leq K \epsilon_N 
\end{align}
and the sum over the rest of the terms is bounded.

To bound the cross terms in \eqref{crossterm}, we begin by defining $\mb R_N^{(ii,jj)}(z) := (z \mb I_N - (\mb X_N - \mb e_i' X_{ii} \mb e_i  - \mb e_j' X_{jj} \mb e_j  )  )^{-1}$ and $\beta_{j,i} := (1 - X_{jj} ( \mb e_j' \mb R_N^{(ii,jj)}(z) \mb e_j))^{-1}$.
Applying \eqref{rank1inv} twice leads to 
 \begin{align*}
  \mb R_N(z) \mb e_i &= \mb R^{(ii)}_N(z) \mb e_i \beta_i\\
  &= \beta_i \left(\mb R^{(ii,jj)}_N(z) \mb e_i  - X_{jj} \mb R^{(ii,jj)}_N(z) \mb e_j' \mb e_j \mb R^{(ii,jj)}_N(z)\mb e_i \beta_{j,i} \right)
  \end{align*}
 
%

Applying this expansion to \eqref{crossterm}:
\begin{align*}
 &\sum_{i\not=j>m} \E[(  (\mb u_N^{*} \mb R(z))_i W_{ii} (\widehat{ \mb R}(z) \mb v_N)_i)(  (\mb v_N^{*} \widehat{\mb R}(\overline z))_j W_{jj} ( \mb R(\overline z) \mb u_N)_j)\indicator{\Omega'}] \\
&=\sum_{i\not=j>m} \E[\beta_i 
\left(\mb u_N^{*} \mb R^{(ii,jj)}_N(z) \mb e_i  - X_{jj}\mb u_N^{*} \mb R^{(ii,jj)}_N(z) \mb e_j \mb e_j' \mb R^{(ii,jj)}_N(z)\mb e_i \beta_{j,i} \right)W_{ii} (\widehat{ \mb R}(z) \mb v_N)_i\\
&\left(  (\mb v_N^{*} \widehat{\mb R}(\overline z))_j W_{jj} ( \mb R^{(jj)}(\overline z) \mb u_N)_j\right)\indicator{\Omega'}]\\
%
&=\sum_{i\not=j>m} \E\left[\beta_i  \left( \vphantom{\left( \frac{1}{1 - X_{jj}/(z+ \mb w_{j}^{j*} \mb R^{ii,jj}_N(z) \mb w_{j}^{j})}  \right)} \right.\right.
\left(\mb u_N^{*}\mb R^{(ii,jj)}_N(z)  \right)_i W_{ii} \Big(\widehat{ \mb R}(z) \mb v_N\Big)_i
  (\mb v_N^{*}  \widehat{  \mb R}(\overline z))_j W_{jj} (\mb R^{(jj)}(\overline z) \mb u_N)_j\\
%
%
&- X_{jj}\left( \mb u_N^{*} \mb R^{(ii,jj)}_N(z) \right)_j R^{(ii,jj)}_{ji}(z) \beta_{j,i} W_{ii} (\widehat{ \mb R}(z) \mb v_N)_i\
\left(\mb v_N^{*} \widehat{ \mb R}_N(\overline z)    \right)_j W_{jj} ( \mb R^{(jj)}(\overline z) \mb u_N)_j
%
 \left.\left. \vphantom{\left( \frac{1}{1 - X_{jj}/(z+ \mb w_{j}^{j*} \mb R^{ii,jj}_N(z) \mb w_{j}^{j})}  \right)} \right)  \indicator{\Omega'} \right]\\
\end{align*}

The first term is estimated by:
\begin{align*}
&\sum_{i\not=j>m} \E\left[\beta_i  
\left(\mb u_N^{*}\mb R^{(ii,jj)}_N(z)  \right)_i W_{ii} \Big(\widehat{ \mb R}(z) \mb v_N\Big)_i
\left(\mb v_N^{*} \mb R^{(jj)}_N(\overline z) \right)_j W_{jj} \Big(\widehat{ \mb R}(\overline z) \mb u_N\Big)_j \indicator{\Omega'}\right]\\
&= \sum_{i\not=j>m} \E[W_{jj}] \E\left[
\left(\mb u_N^{*}\mb R^{(ii,jj)}_N(z)  \right)_i W_{ii} \Big(\widehat{ \mb R}(z) \mb v_N\Big)_i
\left(\mb v_N^{*} \mb R^{(jj)}_N(\overline z) \right)_j  \Big(\widehat{ \mb R}(\overline z) \mb u_N\Big)_j \indicator{\Omega'}\right]\\
&+\sum_{i\not=j>m} \E\left[(1 - \beta_i )
\left(\mb u_N^{*}\mb R^{(ii,jj)}_N(z)  \right)_i W_{ii} \Big(\widehat{ \mb R}(z) \mb v_N\Big)_i
\left(\mb v_N^{*} \mb R^{(jj)}_N(\overline z) \right)_j W_{jj} \Big(\widehat{ \mb R}(\overline z) \mb u_N\Big)_j \indicator{\Omega'}\right]\\
\end{align*}
The first term is zero because $ \E[W_{jj}] =0$ and the second converges to zero using \eqref{betabound}.

The second term is estimated by
\begin{align*}
 &\left|\sum_{i\not=j>m} \E\left[\beta_i  X_{jj}\left( \mb u_N^{*} \mb R^{(ii,jj)}_N(z) \right)_j R^{(ii,jj)}_{ji}(z) \beta_{j,i} 
W_{ii} (\widehat{ \mb R}(z) \mb v_N)_i
\left(\mb v_N^{*} \mb R^{(jj)}_N(\overline z)    \right)_j W_{jj} (\widehat{ \mb R}(\overline z) \mb u_N)_j \indicator{\Omega'} \right] \right| \leq \\
&\left| \E\left[ \frac{1}{\sqrt{N}} \sum_{i>m}  \beta_i  W_{ii}     (\widehat{ \mb R}(z) \mb v_N)_i
\sum_{j\not=i>m}  W_{jj} \left( \mb u_N^{*} \mb R^{(ii,jj)}_N(z) \right)_j R^{(ii,jj)}_{ji}(z) \beta_{j,i}
\left(\mb v_N^{*} \mb R^{(jj)}_N(\overline z)    \right)_j W_{jj} (\widehat{ \mb R}(\overline z) \mb u_N)_j \indicator{\Omega'} \right]\right|\\
\end{align*}
Then we proceed as before, using  \eqref{betabound} to replace $\beta_i$ with 1 and then using independence of $W_{ii}$ and that its expectation is $0$.

\end{proof}


\begin{thebibliography}{99}
\bibitem{AGZ} Anderson G.W., Guionnet A., and Zeitouni O. {\em An Introduction to Random Matrices}, Cambridge Studies in Advanced Mathematics 118, 
(2010), Cambridge University Press, New York.

\bibitem{B} Z.D. Bai, {\em Methodologies in spectral analysis of large-dimensional random 
matrices, a review}. Statist. Sinica 9,  (1999), 611-677.

\bibitem{BP} Z.D. Bai, Z. and G.M. Pan,  {\em Limiting behavior of eigenvectors of large wigner matrices}, Journal of 
Statistical Physics  146, No. 3 (2012), 519-549.


\bibitem{BY} Z.D. Bai and J. Yao, {\em Central limit theorems for eigenvalues in a spiked population model}, Ann. I.H.P.-Prob.et Stat. 44, (2008), 447-474.


\bibitem{BBP} J. Baik, G. Ben Arous, and S. P\'ech\'e,  {\em Phase transition of the largest eigenvalue for non-null complex sample covariance matrices}, 
Ann. Probab. 33, (2005), 1643-1697.

\bibitem{BS} J. Baik  and J.W. Silverstein,  {\em Eigenvalues of large sample covariance matrices of spiked population models}, J. of Multi. Anal. 97, 
(2006), 1382-1408.

\bibitem{BW1} J. Baik and D. Wang, {\em On the largest eigenvalue of a Hermitian random matrix model with spiked external source I. Rank one case},
available at http://arxiv.org/1010.4604.

\bibitem{BW2} J. Baik and D. Wang, {\em On the largest eigenvalue of a Hermitian random matrix model with spiked external source II. Higher rank case},
available at http://arxiv.org//1104.2915.

\bibitem{BG} G. Ben Arous and A. Guionnet,  {\em Wigner matrices}, in Oxford Handbook on Random Matrix Theory, edited by 
Akemann G., Baik J. and Di Francesco P.,  2011, Oxford University Press, New York.

\bibitem{BR} F. Benaych-Georges  and R. Rao, {\em The eigenvalues and eigenvectors of finite, low rank perturbations of large random matrices},
Adv. Math., 227, No. 1, (2011), 494-521.

\bibitem{BGM} F. Benaych-Georges, A. Guionnet, and M. Maida,  {\em Fluctuations of the extreme eigenvalues of finite rank deformations of random 
matrices}, Elec. J. Probab., 16, (2011), 1621-1662.

\bibitem{BGM1} F. Benaych-Georges, A. Guionnet, and M. Maida, {\em Large deviations of the extreme eigenvalues of random deformations of matrices},
to appear in Probab. Theory Related Fields., available at http://arxiv.org/abs/1009.0135v3.

\bibitem{Bill} Billingsley, P. {\em   Probability and Measure}, Wiley Series in Probability and Mathematical Statistics,  
(1995), John Wiley \& Sons Inc., New York.


\bibitem{BV1} A. Bloemendal and  B. Vir\'{a}g, {\em Limits of spiked random matrices I}, available at arXiv:1011.1877.

\bibitem{BV2} A. Bloemendal and  B. Vir\'{a}g, {\em Limits of spiked random matrices II}, available at arXiv:1109.3704.

\bibitem{CD} M. Capitaine and C. Donati-Martin, {\em Strong asymptotic freeness of Wigner and Wishart matrices}, Indiana Univ. Math. J. 56, (2007), 
767-804. 

\bibitem{CDF1} M. Capitaine, C. Donati-Martin, and D. F\'eral, {\em The largest eigenvalue of finite rank deformation of large Wigner matrices: 
convergence and non universality of the fluctuations}, Ann. Probab., 37, (1), (2009), 1-47.

\bibitem{CDF} M. Capitaine, C. Donati-Martin, and D. F\'eral,  {\em Central limit theorems for eigenvalues of deformations of Wigner matrices},
Ann. I.H.P.-Prob.et Stat.,  48, No. 1 (2012), 107-133.

\bibitem{CDFF} M. Capitaine, C. Donati-Martin, D. F\'eral, and M. F\'evrier, 
{\em Free convolution with a semi-circular distribution and eigenvalues of spiked deformations of Wigner matrices},
to appear in Elec. J. Probab., available at arXiv:1006.3684.


\bibitem{D} E.B. Davies, {\em The functional calculus}, J. London Math. Soc., 52, (1995), 166-176.


\bibitem{EKYY} L. Erd\"os, A. Knowles, H.-T. Yau, and J. Yin, {\em Spectral statistics of Erd\"os-R\'enyi graphs II: Eigenvalue
spacing and the extreme eigenvalues},
available at arXiv:1103.3869.

\bibitem{EYY} L. Erd\"os, H.-T. Yau, and J. Yin, {\em Rigidity of eigenvalues of generalized Wigner matrices},
available at arXiv:1007.4652.

\bibitem{E} L. Erd\"os, {\em Universality of Wigner random matrices: a survey of recent results}, 
Uspekhi Mat. Nauk 66 (2011), no. 3(399),67-198.

\bibitem{FP} D. F\'eral and S. P\'ech\'e,  {\em The largest eigenvalue of rank one deformation of large Wigner matrices}  Comm. Math. Phys.  272,  
no. 1, (2007), 185-228.

\bibitem{FK} Z. F\"uredi Z. and J. Koml\'os,  {\em The eigenvalues of random symmetric matrices}, Combinatorica 1, (1981), 233-241.

\bibitem{GS}  Grinberg, V. S., and Sevastyanov, S. V. Value of the Steinitz constant. {\em Functional Analysis and Its Applications 14} (1980), 125-126. 


\bibitem{GZ} A. Guionnet and B. Zegarlinski, {\em Lectures on logarithmic Sobolev inequalities}, Seminaire de Probabilit\'{e}s XXXVI, 
Lecture Notes in Mathematics 1801, (2003), Springer, Paris.


\bibitem{HS} B. Helffer and J. Sj\"{o}strand,  {\em Equation de Schr\"{o}dinger avec champ magnetique et equation de Harper}, Schr\"{o}dinger
Operators, Lecture Notes in Physics 345, 118-197, (eds. H. Holden and A. Jensen)  (1989), Springer, Berlin.

\bibitem{J} K. Johansson,  {\em Universality for certain Hermitian Wigner matrices under weak moment conditions}, 
48, No.1, (2012), 47-79.

\bibitem {Jo} I.M. Johnstone, {\em On the distribution of the largest eigenvalue in principal components
analysis}, Ann. Stat. 29, (2001), 295-327.


\bibitem{KKP} A. Khorunzhy, B. Khoruzhenko  and L. Pastur,  {\em Asymptotic properties of large random matrices 
with independent entries}, J. Math. Phys. 37, (1996), 5033-5060.

\bibitem{KYlsc} A. Knowles and J. Yin, {\em The isotropic semicircle law and deformation of Wigner matrices}, 
available at arXiv:1110.6449.

\bibitem{KY} A. Knowles and J. Yin, {\em Eigenvector distribution of Wigner matrices}, 
available at arXiv:1102.0057.

\bibitem{KYout} A. Knowles and J. Yin, {\em The Outliers of a Deformed Wigner Matrix}, 
available at arXiv:1207.5619.

\bibitem{LP1} A. Lytova and L. Pastur, {\em Fluctuations of matrix elements of regular functions of
              Gaussian random matrices}, J. Stat. Phys., 134, (2009), 147-159.

\bibitem{LP2} A. Lytova and L. Pastur, {\em Non-Gaussian limiting laws for the entries of regular functions of
              the Wigner matrices}, available at arXiv:1103.2345.

\bibitem{L} A. Lytova,  {\em On Non-Gaussian Limiting Laws for the Certain
      Statistics of the Wigner Matrices}, available at arXiv.1201.3027.


\bibitem{Mai}  M. Maida, {\em Large deviations for the largest eigenvalue of rank one deformations of Gaussian ensembles}, Electron. J. Probab. 12, 
(2007), 1131-1150.

\bibitem{M} M.L. Mehta, {\em Random Matrices}, (1991) New York, Academic Press.

\bibitem{ORS} S. O'Rourke, D. Renfrew, and A. Soshnikov, {\em On fluctuations of matrix entries of regular functions of Wigner matrices
with non-identically distributed entries}, to appear in J. Theor. Probab., available at arXiv:1104.1663v3.

\bibitem{Paul} D. Paul,  {\em Asymptotics of sample eigenstructure for a large dimensional spiked covariance model},  Statist. Sinica 17, no.4, 
(2007), 
1617-1642.

\bibitem{P} S. P\'ech\'e, {\em The largest eigenvalue of small rank perturbations of Hermitian random matrices}.  Probab. Theory Related Fields  134,  
no. 1, (2006), 127-173.

\bibitem{PRS1} A. Pizzo, D. Renfrew, and A. Soshnikov,  {\em Fluctuations of matrix entries of regular functions of Wigner matrices}, 
J. Stat. Phys. 146,
No. 3, (2012), 550-591.

\bibitem{PRS} A. Pizzo, D. Renfrew, and A. Soshnikov,  {\em On finite rank deformations of Wigner matrices}, to appear in Annales de L'Institut 
Henri Poincar\'{e} Probabilit\'{e}s et Statistiques, available at arXiv:1103.3731 [math.PR] v.4.


\bibitem{Shc} M. Shcherbina, {\em Central limit theorem for linear eigenvalue statistics of Wigner and sample covariance random matrices}, 
Journal of Mathematical Physics, Analysis, Geometry, 7, No. 2, (2011), 176--192.

\bibitem{Sh} M. Shcherbina,  letter from March 1, 2011.


\bibitem{Shch} M. Shcherbina,  B. Tirozzi, {\em Central limit theorem for fluctuations of linear eigenvalue statistics of large random graphs. 
Diluted regime}, available at arXiv.1111.5492.

\bibitem{S} A. Soshnikov, {\em Universality at the edge of the spectrum in Wigner random matrices}, 
Commun. Math. Phys., 207, (1999), 697-733.


\bibitem{Steinitz} Steinitz, E. {\em Bedingt konvergente Reihen und konvexe Systeme}, Journal f\"{u}r die reine und 
angewandte Mathematik (Crelle's Journal) 143, (1913), 128-175. 


\bibitem{T} T. Tao, {\em Outliers in the spectrum of iid matrices with bounded rank perturbations},
to appear in Probab. Theory  Rel. Fields, available at arXiv:1012.4818v2.

\bibitem{TV} T. Tao and V. Vu, {\em Random matrices: Universal properties of eigenvectors}, available at arXiv:1103.2801.

\bibitem{TV1} T. Tao and V. Vu, {\em Random matrices: Universality of local eigenvalue statistics up to the edge},
Commun. Math. Phys., 298, (2010), 549-572.

\bibitem{TW1} C. Tracy and H. Widom, {\em Level spacing distributions and the Airy kernel}, Commun. Mathematical Physics 159, (1994) 151-174.

\bibitem{TW2} C. Tracy and H. Widom, {\em On orthogonal and symplectic matrix ensembles}, Commun. Mathematical Physics 177, (1996) 727-754.


\end{thebibliography}
\end{document}